\documentclass[leqno]{article}
\usepackage{amsmath,amsthm,bm,colortbl,graphicx,mathrsfs,afterpage,amssymb,url}
\pdfoutput=1
\theoremstyle{plain}
\newtheorem{theorem}{Theorem}
\newtheorem{lemma}[theorem]{Lemma}
\newtheorem{corollary}[theorem]{Corollary}
\theoremstyle{remark}
\newtheorem*{remark}{Remark} 
\def\P{{\rm P}} 
\def\E{{\rm E}} 
\def\T{\textsf{T}} 
\def\emp{\varnothing} 
\allowdisplaybreaks
\makeatletter 
\def\@biblabel#1{\hspace*{-\labelsep}}
\makeatother

\title{Parrondo games with spatial dependence \break and a related spin system}
\author{S. N. Ethier\thanks{}
}
\author{S. N. Ethier\thanks{Partially supported by a grant from the Simons Foundation (209632).  Also supported by a Korean Federation of Science and Technology Societies grant funded by the Korean Government (MEST, Basic Research Promotion Fund).}\\
\begin{small}University of Utah\end{small}\\
\begin{small}Department of Mathematics\end{small}\\
\begin{small}155 South 1400 East, JWB 233\end{small}\\
\begin{small}Salt Lake City, UT 84112, USA\end{small}\\
\begin{small}e-mail: \url{ethier@math.utah.edu}\end{small}
\and
Jiyeon Lee\thanks{Supported by the Basic Science Research Program through the National Research Foundation of Korea (NRF) funded by the Ministry of Education, Science and Technology (2011-0005982).}\\
\begin{small}Yeungnam University\end{small}\\
\begin{small}Department of Statistics\end{small}\\
\begin{small}214-1 Daedong, Kyeongsan\end{small}\\
\begin{small}Kyeongbuk 712-749, South Korea\end{small}\\
\begin{small}e-mail: \url{leejy@yu.ac.kr}\end{small}}
\date{}

\begin{document}
\maketitle

\begin{abstract}
Toral introduced so-called cooperative Parrondo games, in which there are $N\ge3$ players arranged in a circle.  At each turn one player is randomly chosen to play.  He plays either game $A$ or game $B$, depending on the strategy.  Game $A$ results in a win or loss of one unit based on the toss of a fair coin.  Game $B$ results in a win or loss of one unit based on the toss of a biased coin, with the amount of the bias depending on whether none, one, or two of the player's two nearest neighbors have won their most recent games.  Game $A$ is fair, so the games are said to exhibit the Parrondo effect if game $B$ is losing or fair and the random mixture $C:=(1/2)(A+B)$ is winning.  With $\mu^N_B$ (resp., $\mu^N_C$) denoting the mean profit per turn to the ensemble of $N$ players always playing game $B$ (resp., $C$), we give sufficient conditions for $\lim_{N\to\infty}\mu^N_B$ to exist and show that $\lim_{N\to\infty}\mu^N_C$ nearly always exists, with the limits expressible in terms of a parameterized spin system on the one-dimensional integer lattice.  For a particular choice of the parameters, we show that the Parrondo effect (i.e., $\mu_B^N\le0$ and $\mu_C^N>0$) is present in the $N$-player model if and only if $N$ is even. For the same choice of the parameters, we show that, with a suitable interpretation and for certain initial distributions, the Parrondo effect is present in the spin system if and only if $N$ is even, $N$ being the number of consecutive players whose collective profit is tracked.\smallskip\par
\end{abstract}

\section{Introduction}\label{intro}

In Toral's (2001) \textit{cooperative Parrondo games}, there are $N\ge3$ players labeled from 1 to $N$ and arranged in a circle in clockwise order.  At each turn, one player is chosen at random to play.  Call him player $i$.  He plays either game $A$ or game $B$, depending on the strategy.  In game $A$ he tosses a fair coin.  In game $B$ he tosses a $p_0$-coin (i.e., $p_0$ is the probability of heads) if his neighbors $i-1$ and $i+1$ are both losers, a $p_1$-coin if $i-1$ is a loser and $i+1$ is a winner, a $p_2$-coin if $i-1$ is a winner and $i+1$ is a loser, and a $p_3$-coin if $i-1$ and $i+1$ are both winners.  (Because of the circular arrangement, player 0 is player $N$ and player $N+1$ is player 1.) A player's status as winner or loser depends on the result of his most recent game.  Of course, the player of either game wins one unit with heads and loses one unit with tails.  Under these assumptions, the model has an integer parameter $N\ge3$ and four probability parameters $p_0,p_1,p_2,p_3\in[0,1]$.  (The model described in the abstract is the special case in which $p_1=p_2$.)  Game $A$ is fair, so the games are said to exhibit the \textit{Parrondo effect} if game $B$ is losing or fair and the random mixture $C:=(1/2)(A+B)$ (i.e., toss a fair coin to determine which game to play) is winning.  Toral used simulation to find a case in which the Parrondo effect appears, thereby providing a new example of \textit{Parrondo's paradox} (Harmer and Abbott 2002, Abbott 2010).  

Extending the work of Mihailovi\'c and Rajkovi\'c (2003) (see also Xie et al.\ 2011), Ethier and Lee (2012) showed how to compute $\mu^N_B$ (resp., $\mu^N_C$), the mean profit per turn to the ensemble of $N$ players always playing game $B$ (resp., $C$), for $3\le N\le 19$.  Their numerical results suggested that these means converge as $N\to\infty$, and that the Parrondo effect (i.e., $\mu^N_B\le0$ and $\mu^N_C>0$) is present in the limit for a set of parameter vectors having nonzero volume.  In the present paper we give sufficient conditions for $\lim_{N\to\infty}\mu^N_B$ to exist and show that $\lim_{N\to\infty}\mu^N_C$ nearly always exists.  Further, we show that the limiting values are expressible in terms of a parameterized interacting particle system, or spin system, on the one-dimensional integer lattice $\bf Z$.  For a particular choice of the parameters, namely $p_0=1$, $p_1=p_2\in(1/2,1)$, and $p_3=0$, we show that the Parrondo effect is present in the $N$-player model if and only if $N$ is even.  For the same choice of the parameters, we show that, with a suitable interpretation and for certain initial distributions, the Parrondo effect is present in the spin system if and only if $N$ is even, $N$ being the number of consecutive players whose collective profit is tracked. 

Section~\ref{Markov} describes the $N$-player model and the associated discrete-time Markov chain.  Section~\ref{SLLN} establishes a strong law of large numbers (SLLN) for the sequence of profits to the ensemble of $N$ players playing game $B$, giving several formulas for the mean parameter.  Section~\ref{p0=1,p3=0} treats the special case in which we can confirm the Parrondo effect for all even $N\ge4$.    Section~\ref{SLLN-cont-time} describes what one might mean by Parrondo's paradox in continuous time, and establishes an SLLN for the sequence of profits to the ensemble of $N$ players playing game $B$, assuming a continuous-time Markov chain model.  Section~\ref{spin} introduces the related spin system and derives its basic properties, giving sufficient conditions for ergodicity.  Section~\ref{SLLN-spin} establishes an SLLN for the sequence of profits to one of the infinitely many players playing game $B$, assuming the spin system model.  Section~\ref{p0=1,p3=0-spin} examines the special case of Section~\ref{p0=1,p3=0} in the context of the spin system.  Section~\ref{open} concludes with a few open problems that we hope will be of interest to specialists in interacting particle systems.

\section{The discrete-time Markov chain}\label{Markov}

Let us define the Markov chain, introduced by Mihailovi\'c and Rajkovi\'c (2003), that keeps track of the status (loser or winner, 0 or 1) of each of the $N$ players.  It depends on an integer parameter $N\ge3$ and four probability parameters $p_0,p_1,p_2,p_3\in[0,1]$.  Its state space is the product space 
$$
\Sigma:=\{\bm x=(x_1,x_2,\ldots,x_N): x_i\in\{0,1\}{\rm\ for\ }i=1,\ldots,N\}=\{0,1\}^N
$$
with $2^N$ states.  Let $m_i(\bm x):=2x_{i-1}+x_{i+1}$, or, in other words, $m_i(\bm x)$ is the integer (0, 1, 2, or 3) whose binary representation is $(x_{i-1}\,x_{i+1})_2$;  of course, $x_0:=x_N$ and $x_{N+1}:=x_1$.  Also, let $\bm x^i$ be the element of $\Sigma$ equal to $\bm x$ except at the $i$th component.  For example, $\bm x^1:=(1-x_1,x_2,x_3,\ldots,x_N)$.  The one-step transition matrix $\bm P$ for this Markov chain has the form
\begin{equation}\label{x to x^i}
P(\bm x,\bm x^i):=\begin{cases}N^{-1}p_{m_i(\bm x)}&\text{if $x_i=0$,}\\N^{-1}q_{m_i(\bm x)}&\text{if $x_i=1$,}\end{cases}\qquad i=1,\ldots,N,\;\bm x\in\Sigma,
\end{equation}
\begin{equation}\label{x to x}
P(\bm x,\bm x):=N^{-1}\bigg(\sum_{i:x_i=0}q_{m_i(\bm x)}+\sum_{i:x_i=1}p_{m_i(\bm x)}\bigg),\qquad \bm x\in\Sigma,
\end{equation}
where $q_m:=1-p_m$ for $m=0,1,2,3$ and empty sums are 0, and $P(\bm x,\bm y)=0$ otherwise. 

If $p_0,p_1,p_2,p_3\in(0,1)$, then the Markov chain is irreducible and aperiodic, but this assumption is unnecessarily restrictive.  Instead we will assume that $p_0,p_1,p_2,p_3\in[0,1]$ and that \textit{the Markov chain is ergodic} (i.e., there is a unique stationary distribution and the distribution at time $n$ converges to it as $n\to\infty$, regardless of the initial distribution).  What exactly does this involve?  The answer, which is not entirely intuitive, is provided by the following lemma.  We denote the state comprising all $0$s by $\bm0$ and the state comprising all $1$s by $\bm1$.

\begin{lemma}\label{ergodicity-P}
The Markov chain in $\Sigma$ with one-step transition matrix $\bm P$ given by (\ref{x to x^i}) and (\ref{x to x}), where $N\ge3$ and $p_0,p_1,p_2,p_3\in[0,1]$, has the following behavior.

\emph{(a)} If $p_0,p_3\in(0,1)$, then the chain is irreducible and aperiodic, with the following two exceptions.  If $N=3$ and $p_1=p_2=0$, then $\{000,001,010,100\}$ is closed, irreducible, and aperiodic, and the other four states are transient.  If $N=3$ and $p_1=p_2=1$, then $\{011,101,110,111\}$ is closed, irreducible, and aperiodic, and the other four states are transient.

\emph{(b)} Suppose $p_0=1$ and $p_3\in(0,1)$.  Then state $\bm0$ is transient, and $\Sigma-\{\bm0\}$ is closed, irreducible, and aperiodic, with the following two exceptions.  If $p_1=p_2=1$, then all states are transient except those without adjacent $0$s, and the set of such states is closed, irreducible, and aperiodic.  If $N$ is divisible by $3$ and $p_1=p_2=0$, then states $001\cdots001$, $010\cdots010$, and $100\cdots100$ are absorbing, and all other states are transient.

\emph{(c)} Suppose $p_0=0$ and $p_3\in[0,1)$.  Then state $\bm0$ is absorbing, and all other states are transient, with the following two exceptions.  If $N$ is divisible by $3$, $p_3=0$, and $p_1=p_2=1$, then there are four absorbing states, namely $\bm0$, $011\cdots011$, $101\cdots101$, and $110\cdots110$, and all other states are transient.  If $N=3$, $p_3>0$, and $p_1=p_2=1$, then $\bm0$ is absorbing, $\{011,101,110,111\}$ is closed, irreducible, and aperiodic, and the other three states are transient.

\emph{(d)} Suppose $p_0\in(0,1)$ and $p_3=0$.  Then state $\bm1$ is transient, and $\Sigma-\{\bm1\}$ is closed, irreducible, and aperiodic, with the following two exceptions.  If $p_1=p_2=0$, then all states are transient except those without adjacent $1$s, and the set of such states is closed, irreducible, and aperiodic.  If $N$ is divisible by $3$ and $p_1=p_2=1$, then states $011\cdots011$, $101\cdots101$, and $110\cdots110$ are absorbing, and all other states are transient.

\emph{(e)} Suppose $p_0\in(0,1]$ and $p_3=1$.  Then state $\bm1$ is absorbing, and all other states are transient, with the following two exceptions.  If $N$ is divisible by $3$, $p_0=1$, and $p_1=p_2=0$, then there are four absorbing states, namely $\bm1$, $001\cdots001$, $010\cdots010$, and $100\cdots100$, and all other states are transient.  If $N=3$, $p_0<1$, and $p_1=p_2=0$, then $\bm1$ is absorbing, $\{000,001,010,100\}$ is closed, irreducible, and aperiodic, and the other three states are transient.

\emph{(f)} Suppose $p_0=1$ and $p_3=0$.  If $N$ is even, then states $01\cdots01$ and $10\cdots10$ are absorbing, and all other states are transient, with two exceptions described below.  If $N$ is odd, then all states are transient except the $2N$ states in which $0$s and $1$s alternate with the single exception of a pair of adjacent $0$s or a pair of adjacent $1$s, and the set of such states is closed, irreducible, and aperiodic, with the same two exceptions, which are as follows.  If $p_1=p_2=0$, then every state with no adjacent $1$s and no more than two consecutive $0$s is absorbing, and all other states are transient.  If $p_1=p_2=1$, then every state with no adjacent $0$s and no more than two consecutive $1$s is absorbing, and all other states are transient.

\emph{(g)} If $p_0=0$ and $p_3=1$, then states $\bm0$ and $\bm1$ are absorbing, and all other states are transient.

In summary, the Markov chain is ergodic except when $(p_0,p_3)=(0,1)$; or when $N$ is even and $(p_0,p_3)=(1,0)$; or when $(p_0,p_1,p_2,p_3)$ equals $(1,0,0,0)$ or $(1,1,1,0)$;  or when $N$ is divisible by $3$ and either $(p_0,p_1,p_2)=(1,0,0)$ or $(p_1,p_2,p_3)=(1,1,0)$; or when $N=3$ and either $(p_0,p_1,p_2)=(0,1,1)$ or $(p_1,p_2,p_3)=(0,0,1)$.

Furthermore, excluding the exceptions just listed, all of which involve non-uniqueness of stationary distributions, there is a closed, irreducible, aperiodic set of states, and all other states are transient.  Thus, uniqueness of stationary distributions and ergodicity are equivalent here.
\end{lemma}

\begin{proof}
We prove part (f) and leave the remainder of the proof as an exercise for the interested reader.  Consider first the case of $N$ even.  Excluding the two exceptional cases, we notice that either $(i)$ $p_1>0$ and $p_2<1$ or $(ii)$ $p_1<1$ and $p_2>0$.  Start with an arbitrary $\bm x$.  In case $(i)$ we change the entries $x_2,\ldots,x_N$ from left to right as needed to get $01\cdots01$ if $x_1=0$ and $10\cdots10$ if $x_1=1$.  This requires $000\to010$, $001\to011$, $111\to101$, and $110\to100$, all of which are legal moves under $(i)$.  In case $(ii)$ we change the entries $x_1,\ldots,x_{N-1}$ from right to left as needed to get  $01\cdots01$ if $x_N=1$, $10\cdots10$ if $x_N=0$.  This requires $000\to010$, $100\to110$, $111\to101$, and $011\to001$, all of which are legal moves under $(ii)$.  

We turn next to the case of $N$ odd.  Excluding the exceptional cases, again consider cases $(i)$ and $(ii)$ as above.  Let us label the $2N$ states as follows.  A state with alternating 0s and 1s except for a pair of adjacent 0s is labeled by $(0,i)$ if the first 0 of the pair (in clockwise order) occurs at position $i$.  A state with alternating 0s and 1s except for a pair of adjacent 1s is labeled by $(1,i)$ if the first 1 of the pair (in clockwise order) occurs at position $i$.  Then the chain jumps from $(0,i)$ to $(1,i-1)$ and $(1,i+1)$ with probabilities $N^{-1}p_2$ and $N^{-1}p_1$, respectively, and from $(1,i)$ to $(0,i-1)$ and $(0,i+1)$ with probabilities $N^{-1}q_1$ and $N^{-1}q_2$, respectively; as usual, $1-1:=N$ and $N+1:=1$.  In case $(i)$, $p_1>0$ and $q_2>0$, so that the chain can cycle through the $2N$ states in clockwise order, whereas in case $(ii)$, $p_2>0$ and $q_1>0$, so that the chain can cycle through the $2N$ states in counter-clockwise order.  Irreducibility of the set of $2N$ states leads to aperiodicity because $P(\bm x,\bm x)\ge1-2/N>0$ for every such $\bm x$.  To see that all other states are transient, we use the same argument as in the preceding paragraph.  But here the target states are $01\cdots010$ and $10\cdots101$.

Finally, regardless of the parity of $N$, assume $p_1=p_2=0$.  (The case $p_1=p_2=1$ is symmetric.)  Clearly, every state with no adjacent 1s and no more than two consecutive 0s is absorbing.  Starting from an arbitrary $\bm x$, flip each 0 surrounded by 0s, in clockwise order, say.  Then 0s occur only as singletons or pairs.  Then do the same for 1s, so that 1s occur only as singletons or pairs.  Since $q_1=q_2=1$, we can flip one of the 1s in each adjacent pair.  This may create three consecutive 0s; if so, flip the center one.  The resulting state has no adjacent 1s and no more than two consecutive 0s.   
\end{proof}

We have described what is known as game $B$.  The description suggests that its long-term behavior should be invariant under rotation (and, if $p_1=p_2$, reflection) of the $N$ players.  

\begin{lemma}\label{invariance}
Let $G$ be a subgroup of the symmetric group $S_N$.  Let $\bm P$ be the one-step transition matrix for a Markov chain in $\Sigma$ having a unique stationary distribution $\bm\pi$.  For $\bm x=(x_1,\ldots,x_N)\in\Sigma$ and $\sigma\in G$, write $\bm x_\sigma:=(x_{\sigma(1)},\ldots,x_{\sigma(N)})$, and assume that
$P(\bm x_\sigma,\bm y_\sigma)=P(\bm x,\bm y)$ for all $\sigma\in G$ and $\bm x,\bm y\in\Sigma$. 
Then $\pi(\bm x_\sigma)=\pi(\bm x)$ for all $\sigma\in G$ and $\bm x\in\Sigma$.  
\end{lemma}

\begin{proof}
Given $\sigma\in G$, define the distribution ${\bm\pi}_\sigma$ on $\Sigma$ by ${\pi}_\sigma(\bm x):=\pi(\bm x_\sigma)$.  Then
\begin{equation*}
\pi_\sigma(\bm y)
=\sum_{\bm x\in\Sigma}\pi(\bm x)P(\bm x,\bm y_\sigma)=\sum_{\bm x\in\Sigma}\pi(\bm x_\sigma)P(\bm x_\sigma,\bm y_\sigma)
=\sum_{\bm x\in\Sigma}\pi_\sigma(\bm x)P(\bm x,\bm y)
\end{equation*}
for all $\bm y\in\Sigma$, hence by the uniqueness of stationary distributions, $\bm\pi_\sigma=\bm\pi$. 
\end{proof}

The lemma applies to our Markov chain (assuming it is ergodic) if $G$ is the subgroup of \textit{cyclic permutations} (or rotations) of $(1,2,\ldots,N)$, that is, the group generated by
\begin{equation}\label{cyclic}
(\sigma(1),\sigma(2),\ldots,\sigma(N)):=(2,3,\ldots,N,1).
\end{equation}
If $p_1=p_2$, then it also applies if $G$ is the subgroup generated by (\ref{cyclic}) and 
the \textit{order-reversing permutation} (or reflection) of $(1,2,\ldots,N)$,
\begin{equation*}
(\sigma(1),\sigma(2),\ldots,\sigma(N)):=(N,N-1,\ldots,2,1).
\end{equation*}
In this case $G$ is known as the \textit{dihedral group} of order $2N$.  The calculation that justifies these conclusions is
\begin{eqnarray}\label{invariance_property}
P(\bm x_\sigma,(\bm x^i)_\sigma)&=&P(\bm x_\sigma,(\bm x_\sigma)^{\sigma^{-1}(i)})\\
&=&\begin{cases}N^{-1}p_{m_{\sigma^{-1}(i)}(\bm x_\sigma)}&\text{if $(\bm x_\sigma)_{\sigma^{-1}(i)}=0$}\nonumber\\
N^{-1}q_{m_{\sigma^{-1}(i)}(\bm x_\sigma)}&\text{if $(\bm x_\sigma)_{\sigma^{-1}(i)}=1$}\end{cases}\nonumber\\
&=&\begin{cases}N^{-1}p_{m_i(\bm x)}&\text{if $x_i=0$}\\N^{-1}q_{m_i(\bm x)}&\text{if $x_i=1$}\end{cases}\nonumber\\
&=&P(\bm x,\bm x^i)\nonumber
\end{eqnarray}
for $i=1,\ldots,N$ and all $\bm x\in\Sigma$.

We conclude this section with an application of Lemma~\ref{invariance} that will be useful later.

\begin{corollary}\label{symm}
Assume that the Markov chain in $\Sigma$ with one-step transition matrix $\bm P$ given by (\ref{x to x^i}) and (\ref{x to x}) is ergodic with unique stationary distribution $\bm\pi$, and denote by $\pi_{1,3}$ its $1,3$ two-dimensional marginal.  If also $p_1=p_2$, then $\pi_{1,3}(0,1)=\pi_{1,3}(1,0)$.
\end{corollary}

\begin{proof}
By Lemma~\ref{invariance} with $G$ being the dihedral group, $\pi(x_1,x_2,x_3,x_4,\ldots,x_N)=\pi(x_3,x_2,x_1,x_N,\ldots,x_4)$ for all $\bm x\in\Sigma$.  Now sum over $x_2$ and $x_4,\ldots,x_N$ to obtain the desired result.
\end{proof}

\section{SLLN}\label{SLLN}

The strong law of large numbers of Ethier and Lee (2009) applies not to the Markov chain of Section~\ref{Markov} but to a slightly more informative Markov chain.  The new state space is
$\Sigma^*:=\Sigma\times\{1,2,\ldots,N\}$ and the process is in state $(\bm x,i)$ if $\bm x$ describes the status of each player and $i$ is the next player to play.  The transition matrix $\bm P^*$ has the form
\begin{equation}\label{P^*part1}
P^*((\bm x,i),(\bm x^i,j)):=\begin{cases}N^{-1}p_{m_i(\bm x)}&\text{if $x_i=0$,}\\N^{-1}q_{m_i(\bm x)}&\text{if $x_i=1$,}\end{cases}
\end{equation}
\begin{equation}\label{P^*part2}
P^*((\bm x,i),(\bm x,j)):=\begin{cases}N^{-1}q_{m_i(\bm x)}&\text{if $x_i=0$,}\\N^{-1}p_{m_i(\bm x)}&\text{if $x_i=1$,}\end{cases}
\end{equation}
for all $(\bm x,i)\in\Sigma^*$ and $j=1,2,\ldots,N$, where $q_m:=1-p_m$ for $m=0,1,2,3$, and $P^*((\bm x,i),(\bm y,j))=0$ otherwise.

\begin{lemma}\label{ergodicity-P^*}
Assume that the Markov chain in $\Sigma$ with one-step transition matrix $\bm P$ given by (\ref{x to x^i}) and (\ref{x to x}) is ergodic (see Lemma~\ref{ergodicity-P} for necessary and sufficient conditions) with unique stationary distribution $\bm\pi$.  Then there exists $S^*\subset\Sigma^*$ such that, with respect to $\bm P^*$ of (\ref{P^*part1}) and (\ref{P^*part2}), $S^*$ is closed, irreducible, and aperiodic, and all states in $\Sigma^*-S^*$ are transient.  In particular, the Markov chain with one-step transition matrix $\bm P^*$ is ergodic.  Its unique stationary distribution $\bm\pi^*$ is given by $\pi^*(\bm x,i):=N^{-1}\pi(\bm x)$.
\end{lemma}

\begin{proof}
By Lemma~\ref{ergodicity-P}, there exists $S\subset\Sigma$ such that, with respect to $\bm P$, $S$ is closed, irreducible, and aperiodic, and all states in $\Sigma-S$ are transient.  We claim that we can take $S^*:=S\times\{1,2,\ldots,N\}$.  Given $\bm x,\bm y\in S$ and $i,j\in\{1,2,\ldots,N\}$, we must show that the $\bm P^*$-chain can get from $(\bm x,i)$ to $(\bm y,j)$.  There are two (possibly overlapping) cases:  $P^*((\bm x,i),(\bm x^i,k))>0$ for all $k$ or $P^*((\bm x,i),(\bm x,k))>0$ for all $k$.  (These probabilities do not depend on $k$ and they sum to $N^{-1}$ for fixed $k$, so at least one of them must be positive.)  In the first case,   it suffices to note that the $\bm P$-chain can get from $\bm x^i$ to $\bm y$.  In the second case, it suffices to note that the $\bm P$-chain can get from $\bm x$ to $\bm y$.  This implies the stated irreducibility.  The other properties follow in a similar way.  Since $\bm\pi$ is stationary for $\bm P$, we have
\begin{eqnarray*}
\pi^*(\bm x,j)&=&N^{-1}\pi(\bm x)=N^{-1}\sum_{i=1}^N\pi(\bm x^i)P(\bm x^i,\bm x)+N^{-1}\pi(\bm x)P(\bm x,\bm x)\\
&=&\sum_{i=1}^N\pi^*(\bm x^i,i)P^*((\bm x^i,i),(\bm x,j))+\sum_{i=1}^N\pi^*(\bm x,i)P^*((\bm x,i),(\bm x,j)),
\end{eqnarray*}
so $\bm\pi^*$, defined as in the statement of the lemma, is stationary for $\bm P^*$.
\end{proof}

Notice also that the profit corresponding to each nonzero entry of $\bm P^*$ is equal to $\pm1$, so the SLLN holds and there are several formulas for the mean parameter, as we now show.

\begin{theorem}\label{SLLN-thm}
Assume that the Markov chain in $\Sigma$ with one-step transition matrix $\bm P$ given by (\ref{x to x^i}) and (\ref{x to x}) is ergodic (see Lemma~\ref{ergodicity-P} for necessary and sufficient conditions) with unique stationary distribution $\bm\pi$.  Let $\{(\bm X(n),I(n))\}_{n\ge0}$ be the Markov chain in $\Sigma^*$ described above, with an arbitrary initial distribution.  Define 
$$
\xi_n:=w((\bm X(n-1),I(n-1)),(\bm X(n),I(n))), \qquad n\ge1,
$$
where the payoff function $w$ is 1 for a win and $-1$ for a loss, determined by whether the corresponding entry of $\bm P^*$ is of the form $N^{-1}p_m$ or $N^{-1}q_m$.  Let $S_n:=\xi_1+\cdots+\xi_n$ for each $n\ge1$.  Then $n^{-1}S_n\to\mu^N$ {\rm a.s.}, where the mean parameter $\mu^N$ can be expressed in terms of $\bm\pi$ as 
\begin{equation}\label{mu1}
\mu^N={1\over N}\sum_{\bm x\in\Sigma}\pi(\bm x)\sum_{i=1}^N [p_{m_i(\bm x)}-q_{m_i(\bm x)}],
\end{equation}
in terms of certain two-dimensional marginals $\pi_{1,3}=\pi_{2,4}=\cdots=\pi_{N-1,1}=\pi_{N,2}$ of $\bm\pi$ as
\begin{eqnarray}\label{mu2}
\quad\mu^N&=&\sum_{u=0}^1\sum_{v=0}^1\pi_{1,3}(u,v)(p_{2u+v}-q_{2u+v})\\
&=&2[\pi_{1,3}(0,0)p_0+\pi_{1,3}(0,1)p_1+\pi_{1,3}(1,0)p_2+\pi_{1,3}(1,1)p_3]-1,\nonumber
\end{eqnarray}
or in terms of the one-dimensional marginals $\pi_1=\pi_2=\cdots=\pi_N$ of $\bm\pi$ as
\begin{equation}\label{mu3}
\mu^N=\pi_1(1)-\pi_1(0)=2\pi_1(1)-1.
\end{equation}
\end{theorem}

\begin{remark}
Another formula for $\mu^N$, better suited to numerical computation, was given by Ethier and Lee (2012).
\end{remark}

\begin{proof}
Theorem~1 of Ethier and Lee (2009), applied to the Markov chain in $\Sigma^*$ with one-step transition matrix $\bm P^*$ (restricted to $S^*$ of Lemma~\ref{ergodicity-P^*} to ensure irreducibility and aperiodicity), gives the SLLN with $\mu^N=\bm\pi^*\dot{\bm P^*}\bm1$, where $\dot{\bm P^*}$ is $\bm P^*$ with each $q_m$ replaced by $-q_m$ and $\bm 1:=(1,1,\ldots,1)^\T$, and this implies (\ref{mu1}).  Actually, this works directly if the initial distribution is concentrated on $S^*$.  If it is not, the fact that each state in $\Sigma^*-S^*$ is transient implies that the chain reaches $S^*$ with probability 1, so the SLLN is unaffected.  

Next, using (\ref{mu1}) and the rotation invariance property (Lemma~\ref{invariance}), we have
\begin{eqnarray*}
\mu^N&=&{1\over N}\sum_{\bm x\in\Sigma}\pi(\bm x)\sum_{i=1}^N [p_{m_i(\bm x)}-q_{m_i(\bm x)}]\\
&=&{1\over N}\sum_{i=1}^N \sum_{u=0}^1\sum_{v=0}^1\pi_{i-1,i+1}(u,v)(p_{2u+v}-q_{2u+v})\\
&=&\sum_{u=0}^1\sum_{v=0}^1\pi_{1,3}(u,v)(p_{2u+v}-q_{2u+v}),
\end{eqnarray*}
where $\pi_{0,2}:=\pi_{N,2}$ and $\pi_{N-1,N+1}:=\pi_{N-1,1}$, which implies (\ref{mu2}).  

Finally, turning to (\ref{mu3}), we notice that the conditional probability that player 2 is a winner, given that players 1 and 3 are losers, is not equal to $p_0$.  Instead we have to look back to the last time player 2 played before we condition on the status of player 1 and that of player 3.  With this in mind, we let $\{(\bm X(n),I(n))\}_{n\in{\bf Z}}$ be a stationary version of the Markov chain in $\Sigma^*$ with time indexed by ${\bf Z}$.  Then
\begin{eqnarray*}
\pi_1(1)&=&\pi_2(1)=\pi\{\bm x: x_2=1\}=\P(X_2(0)=1)\\
&=&\sum_{n=1}^\infty\P(X_2(-n+1)=1,I(-n)=2,I(-n+1)\ne2,\ldots,I(-1)\ne2)\\
&=&\sum_{n=1}^\infty\bigg(1-{1\over N}\bigg)^{n-1}\P(X_2(-n+1)=1,I(-n)=2)\\
&=&\sum_{n=1}^\infty\bigg(1-{1\over N}\bigg)^{n-1}\sum_{u=0}^1\sum_{v=0}^1\P(X_1(-n)=u,X_3(-n)=v)\\
&&\quad{}\cdot\P(X_2(-n+1)=1,I(-n)=2)\mid X_1(-n)=u,X_3(-n)=v)\\
&=&\sum_{n=1}^\infty\bigg(1-{1\over N}\bigg)^{n-1}\sum_{u=0}^1\sum_{v=0}^1\pi_{1,3}(u,v)N^{-1}p_{2u+v}\\
&=&\pi_{1,3}(0,0)p_0+\pi_{1,3}(0,1)p_1+\pi_{1,3}(1,0)p_2+\pi_{1,3}(1,1)p_3,
\end{eqnarray*}
where $\{I(-n)=2,I(-n+1)\ne2,\ldots,I(-1)\ne2\}:=\{I(-1)=2\}$ if $n=1$.
Therefore, (\ref{mu3}) follows from (\ref{mu2}).
\end{proof}

We conclude with an application of the SLLN.  The first conclusion will play a minor role in the next section, and the second conclusion is included for completeness.

\begin{corollary}\label{couple}
Let us refer to the Markov chain in $\Sigma$ with one-step transition matrix $\bm P$ given by (\ref{x to x^i}) and (\ref{x to x}) as the Markov chain with parameter vector $(p_0,p_1,p_2,p_3)$, and let us denote $\mu^N$ of Theorem~\ref{SLLN-thm} by $\mu^N(p_0,p_1,p_2,p_3)$ to emphasize its dependence on the parameter vector.

\emph{(a)} If the Markov chain with parameter vector $(p_0,p_1,p_2,p_3)$ is ergodic, then the Markov chain with parameter vector $(q_3,q_2,q_1,q_0)$ is ergodic, and
$$
\mu^N(p_0,p_1,p_2,p_3)=-\mu^N(q_3,q_2,q_1,q_0).
$$
In particular, if $p_0+p_3=1$ and $p_1+p_2=1$, then $\mu^N(p_0,p_1,p_2,p_3)=0$.

\emph{(b)} If the Markov chain with parameter vector $(p_0,p_1,p_2,p_3)$ is ergodic, then
$$
2\min(p_0,p_1,p_2,p_3)-1\le\mu^N(p_0,p_1,p_2,p_3)\le2\max(p_0,p_1,p_2,p_3)-1.
$$
\end{corollary}

\begin{remark}
$(i)$ Game $A$ is the special case of game $B$ in which $p_0=p_1=p_2=p_3=1/2$ and game $C$ is the equally weighted random mixture of game $A$ and game $B$.  Let $\mu_B^N$ (resp., $\mu_C^N$) denote the mean profit per turn to the ensemble of $N\ge3$ players always playing game $B$ (resp., $C$).  Then
\begin{equation}\label{mu_B,mu_C}
\begin{split}
\mu_B^N&:=\mu^N(p_0,p_1,p_2,p_3),\\
\mu_C^N&:=\mu^N((1/2+p_0)/2,(1/2+p_1)/2,(1/2+p_2)/2,(1/2+p_3)/2).
\end{split}
\end{equation}
We say the \textit{Parrondo effect} is present if $\mu_B^N\le0$ and $\mu_C^N>0$, whereas the \textit{anti-Parrondo} effect is present if $\mu_B^N\ge0$ and $\mu_C^N<0$.  Part (a) implies that the Parrondo effect is present for the parameter vector $(p_0,p_1,p_2,p_3)$ if and only if the anti-Parrondo effect is present for the parameter vector $(q_3,q_2,q_1,q_0)$.  It follows that the ``Parrondo region'' and the ``anti-Parrondo region'' have the same (four-dimensional) volume.

$(ii)$ Part (b) generalizes the obvious identity $\mu^N(p,p,p,p)=2p-1$. 
\end{remark}

\begin{proof}
(a) This was proved by Ethier and Lee (2012) using a coupling argument.   

(b) Here we use another coupling.   Let $U_1,U_2,\ldots$ be i.i.d.\ uniform $(0,1)$ random variables, and let $I_1,I_2,\ldots$ be an independent sequence of i.i.d.\ uniform $\{1,2,\ldots,N\}$ random variables.  Then our Markov chain $\{\bm X(n)\}$ can be constructed in such a way that, for each $n\ge1$, $\bm X(n)$ is a function of $\bm X(n-1)$, $I_n$, and $U_n$ with
$$
S_n:=\sum_{k=1}^n\Big[2\cdot1_{(0,p_{m_{I_k}(\bm X(k-1))}]}(U_k)-1\Big]
$$
representing the total profit after $n$ turns to the ensemble of $N$ players.  With $S_n^p:=\sum_{k=1}^n[2\cdot1_{(0,p]}(U_k)-1]$, we have
$$
S_n^{\min(p_0,p_1,p_2,p_3)}\le S_n\le S_n^{\max(p_0,p_1,p_2,p_3)},\qquad n\ge1,
$$
so the desired result follows from Theorem~\ref{SLLN-thm} and the i.i.d.\ SLLN.
\end{proof}

\section{The case $p_0=1$, $p_3=0$}\label{p0=1,p3=0}

Suppose $p_0=1$ and $p_3=0$.  If we also assume that $p_1=p_2\in(1/2,1)$, then we can confirm the Parrondo effect for all even $N\ge4$.  

\begin{theorem}\label{p0=1,p3=0-thm}
Let $p_0=1$, $p_1=p_2\in(1/2,1)$, and $p_3=0$.  Let $\mu^N_B$ (resp., $\mu^N_C$) denote the mean profit per turn to the ensemble of $N\ge3$ players always playing game $B$ (resp., $C$); cf.\ (\ref{mu_B,mu_C}).  Then $\mu^N_B=0$ for all even $N\ge4$, $\mu^N_B>0$ for all odd $N\ge3$, and $\mu^N_C>0$ for all $N\ge3$.  In particular, the Parrondo effect is present if and only if $N$ is even.
\end{theorem}

\begin{proof}
First, we compute $\mu^N_B$ assuming only $p_0=1$, $p_1,p_2\in[0,1]$, $p_3=0$, and $0<p_1+p_2<2$.

If $N$ is even, then the two states $01\cdots01$ and $10\cdots10$, in which 0s and 1s alternate, are absorbing, and all other states are transient (see Lemma~\ref{ergodicity-P}).  From either of these two states there is a win of one unit if the player chosen to play is a winner (probability 1/2) and a loss of one unit if the player chosen to play is a loser (probability 1/2).  Consequently, the i.i.d.\ SLLN applies and $\mu^N_B=0$, regardless of $p_1$ and $p_2$.

On the other hand, if $N$ is odd, then the set of $2N$ states in which 0s and 1s alternate, with the single exception of a pair of adjacent 0s or a pair of adjacent 1s, is closed, irreducible, and aperiodic, and all other states are transient (see Lemma~\ref{ergodicity-P}).  Let us order the states in this set as follows:  First, the states with adjacent 0s are ordered by the position of the first 0 of the adjacent pair in clockwise order (e.g., $001\cdots01$ is first, $010\cdots10$ is $N$th).  Next, the states with adjacent 1s are ordered by the position of the first 1 of the adjacent pair in clockwise order.  With this ordering, the $2N\times2N$ one-step transition matrix obtained by restricting $\bm P$ to this set has the block form
$$
\bar{\bm P} :={1\over N}\left(\begin{array}{cc}(N-p_1-p_2)\bm I_N&\bm P^{p_1,p_2}\\
\bm P^{q_2,q_1}&(N-q_1-q_2)\bm I_N\end{array}\right),
$$
where $\bm P^{p_1,p_2}$ is the $N\times N$ matrix with each entry of the superdiagonal as well as the $(N,1)$ entry equal to $p_1$, and each entry of the subdiagonal as well as the $(1,N)$ entry equal to $p_2$.  By Lemma~\ref{invariance}, the unique stationary distribution $\bar{\bm\pi}$ for this Markov chain has the form
$$
\bar{\bm\pi}=(\alpha_N,\alpha_N,\ldots,\alpha_N,\beta_N,\beta_N,\ldots,\beta_N), 
$$
so $\bar{\bm\pi}=\bar{\bm\pi}\bar{\bm P}$ and $\bar{\bm\pi}\bm1=1$ result in $\alpha_N=(q_1+q_2)/(2N)$ and $\beta_N=(p_1+p_2)/(2N)$.  Finally, by (\ref{mu3}),
$$
\mu^N_B=2\sum_{k=1}^N\bar{\pi}(2k)-1=2\bigg[{N-1\over2}\,\alpha_N+{N+1\over2}\,\beta_N\bigg]-1
={p_1+p_2-1\over N}.
$$

In particular, this proves the assertions about $\mu^N_B$ when $p_1=p_2\in(1/2,1)$.  It remains to show that $\mu^N_C>0$ under this assumption, regardless of the parity of $N$.

Now $\mu^N_C=\mu^N(3/4,(1/2+p_1)/2,(1/2+p_1)/2,1/4)$, which by Corollary~\ref{couple} is 0 at $p_1=1/2$.  If we could show that this function is increasing in $p_1$, the proof would be complete.  However, despite being very plausible, this monotonicity property appears difficult to prove.  We can prove it computationally for small $N$.  For example,
$$
\mu^N_C=\begin{cases}
(2p_1-1)/5&\text{if $N=3$,}\\
7(2p_1-1)/(53-16p_1+16p_1^2)&\text{if $N=4$,}\\
(2p_1-1)(65-14p_1+14p_1^2)/(423-134p_1+134p_1^2)&\text{if $N=5$.}\end{cases}
$$
But a noncomputational proof is needed.  Fortunately, there is an alternative approach that avoids the monotonicity question.

Let $\pi_{1,3}$ be the $1,3$ two-dimensional marginal of $\bm\pi$ when the probability parameters are $3/4$, $(1/2+p_1)/2$, $(1/2+p_1)/2$, and $1/4$.  We apply Theorem~\ref{SLLN-thm} twice.  By (\ref{mu3}) and Corollary~\ref{symm},
\begin{eqnarray}\label{z-x1}
\mu^N_C&=&\pi_{1,3}(1,0)+\pi_{1,3}(1,1)-[\pi_{1,3}(0,0)+\pi_{1,3}(0,1)]\\
&=&\pi_{1,3}(1,1)-\pi_{1,3}(0,0).\nonumber
\end{eqnarray}
By (\ref{mu2}), Corollary~\ref{symm}, and (\ref{z-x1}),
\begin{eqnarray*}
\mu^N_C&=&\pi_{1,3}(0,0)(1/2)+\pi_{1,3}(0,1)(2p_1-1)+\pi_{1,3}(1,1)(-1/2)\\
&=&(2p_1-1)\pi_{1,3}(0,1)-(1/2)\mu^N_C.
\end{eqnarray*}
Therefore, $\mu^N_C=(2/3)(2p_1-1)\pi_{1,3}(0,1)$, and this is positive by the irreducibility of the Markov chain with the stated probability parameters and the assumption that $p_1>1/2$.
\end{proof}

Under the assumptions of Theorem~\ref{p0=1,p3=0-thm}, we know that $n^{-1}S_n\to\mu^N$ a.s.\ (by the i.i.d.\ SLLN in the case of game $B$ and even $N$, by Theorem~\ref{SLLN-thm} otherwise), hence $n^{-1}\E[S_n]\to\mu^N$.  But the rate of convergence here depends on $N$ and of course on the initial distribution.  This is illustrated in Figure~\ref{fig}, in which $N=12$ and 13, $p_0=1$, $p_1=p_2=3/4$, and $p_3=0$.  We assume an initial distribution that is uniform on $\Sigma$ and use $S_0:=0$ and
\begin{equation}\label{cum-mean-profit}
N^{-1}\E[S_{nN}]={1\over N}\sum_{m=1}^{nN}\bm\pi_0\bm P^{m-1}\dot{\bm P}\bm1,\qquad n=1,2,\ldots,100,
\end{equation}
where $\bm\pi_0$ is the $2^N$-dimensional row vector with every entry equal to $2^{-N}$, $\bm P$ is the $2^N\times2^N$ one-step transition matrix from (\ref{x to x^i}) and (\ref{x to x}), $\dot{\bm P}$ is $\bm P$ with each $q_m$ replaced by $-q_m$, and $\bm1$ is the $2^N$-dimensional column vector of 1s.  We evaluate (\ref{cum-mean-profit}) by recursion, not by matrix arithmetic, because the matrices are rather large.
\begin{figure}[ht]
\centering
\includegraphics[width = 4.8in]{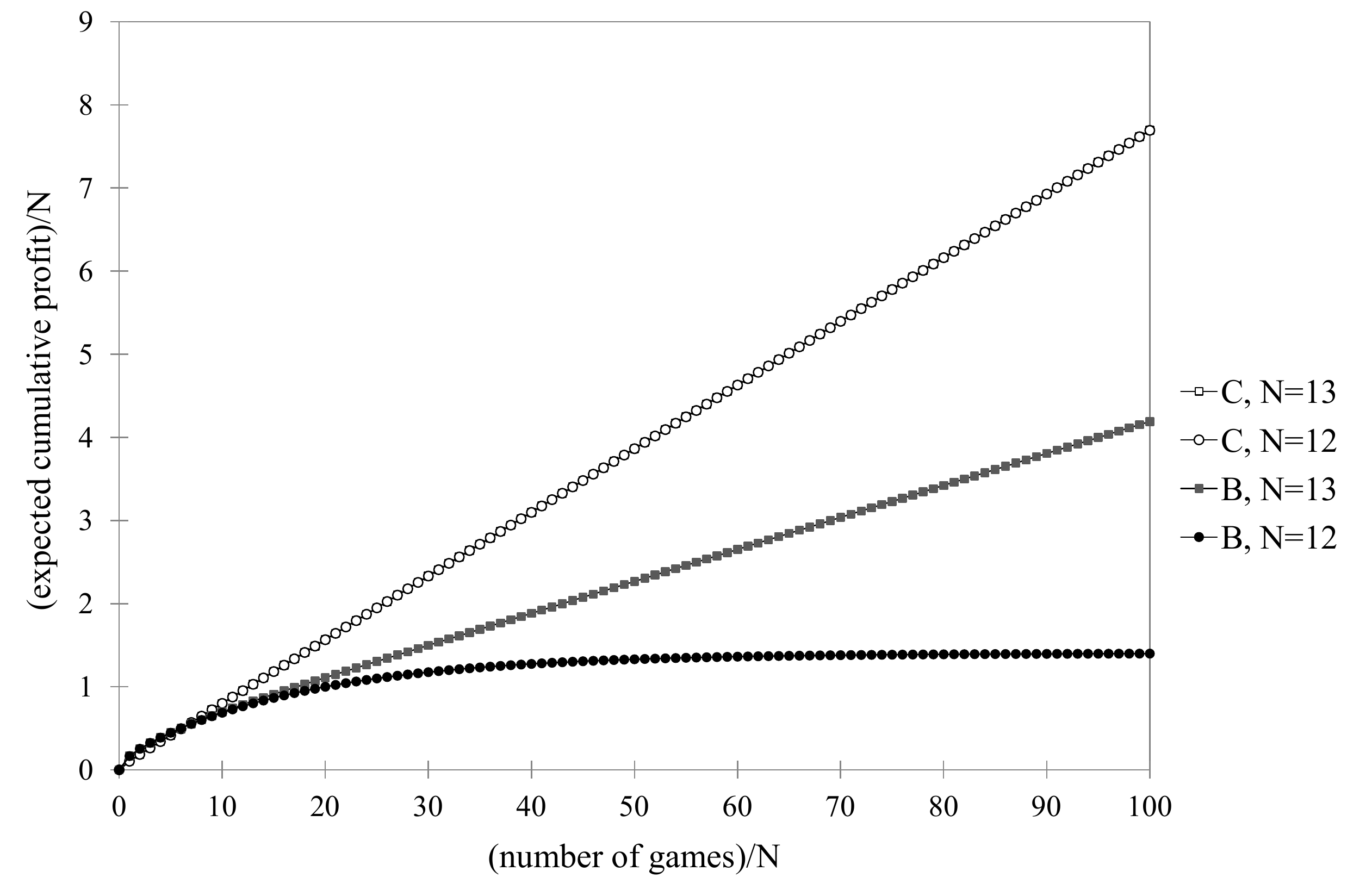}
\caption{\label{fig}Illustration of Theorem~\ref{p0=1,p3=0-thm} in the special cases $N=12$ and $13$ and $p_0=1$, $p_1=p_2=3/4$, and $p_3=0$.  The initial distribution is assumed uniform on $\Sigma$.  The Parrondo effect is present for $N=12$ because $\mu^{12}_B=0$ and $\mu^{12}_C\approx0.0766019$.  It is not present for $N=13$ because $\mu^{13}_B=1/26\approx0.0384615$ and $\mu^{13}_C\approx0.0766021$.  (The open squares are hidden behind the open circles.)}
\end{figure}

\section{SLLN in continuous time}\label{SLLN-cont-time}

The \textit{Parrondo effect} is said to occur if there is a reversal in direction in some system parameter when two similar dynamics are combined.  The term ``combined'' is fairly unambiguous when time is discrete but less so when time is continuous.  The original examples of the paradox all had the following form.  Game $A$ is fair or losing, game $B$ is fair or losing, but the combination of the two games, game $C$, is winning.  Here game $C$ has two possible interpretations.  It could be the random mixture of game $A$ and game $B$, denoted by $C:=\gamma A+(1-\gamma)B$, where $0<\gamma<1$; with probability $\gamma$, game $A$ is played, and with probability $1-\gamma$, game $B$ is played.  Or it could be a nonrandom pattern such as $C:=A^r B^s$; game $A$ is played $r$ times, then game $B$ is played $s$ times.  In these examples, games $A$ and $B$ are controlled by Markov chains in a state space $\Sigma_0$ with one-step transition matrices $\bm P_A$ and $\bm P_B$.  Then game $C$ is controlled by Markov chains in $\Sigma_0$ with one-step transition matrices
\begin{equation}\label{P_C}
\bm P_C:=\gamma\bm P_A+(1-\gamma)\bm P_B\quad\text{and}\quad\bm P_C:=\bm P_A^r\bm P_B^s
\end{equation}
in the random-mixture and nonrandom-pattern cases, respectively.

Now suppose that games $A$ and $B$ are controlled by continuous-time Markov processes in a compact state space $E$ with Feller semigroups $\{T_A(t)\}$ and $\{T_B(t)\}$ on $C(E)$ having generators $\mathscr{L}_A$ and $\mathscr{L}_B$, at least on a domain $D\subset C(E)$.  The analogues of (\ref{P_C}) are
\begin{equation}\label{mixture}
[\gamma T_A(1/n)+(1-\gamma)T_B(1/n)]^{\lfloor nt\rfloor}f\to T_C(t)f
\end{equation}
and, with $\gamma:=r/(r+s)$,
\begin{equation}\label{pattern}
[T_A(\gamma/n)T_B((1-\gamma)/n)]^{\lfloor nt\rfloor}f\to T_C(t)f,
\end{equation}
where $\{T_C(t)\}$ is the Feller semigroup whose generator is assumed to be the closure of $\mathscr{L}_C:=\gamma\mathscr{L}_A+(1-\gamma)\mathscr{L}_B$ acting on $D$.  The limit (\ref{mixture}) is justified by Chernoff's product formula, and the limit (\ref{pattern}) is justified by Trotter's product formula.  So, in continuous time, there is no distinction between the random-mixture and nonrandom-pattern cases, at least in the limit as the time allotted to each game tends to 0.  See Montero (2011) for a different approach to the same issue.

Let us consider whether the SLLN applies to the ensemble's profits in the continuous-time analogue of the $N$-player model in the state space $\Sigma$ or, better yet, $\Sigma^*:=\Sigma\times\{1,2,\ldots,N\}$.  The discrete-time Markov chain in $\Sigma^*$ has one-step transition matrix $\bm P^*$ given by (\ref{P^*part1}) and (\ref{P^*part2}).  We assume that the conditions for ergodicity in Lemma~\ref{ergodicity-P} are satisfied.

The continuous-time analogue of this process has infinitesimal matrix $\bm Q^*:=\bm P^*-\bm I$, and it can be constructed as follows:  Let $\{\bm X^*(n)\}$ denote the discrete-time Markov chain in $\Sigma^*$ and let $\{N(t)\}$ be an independent, rate 1, Poisson process.  Then $\bm Y^*(t):=\bm X^*(N(t))$ defines a continuous-time Markov chain $\{\bm Y^*(t)\}$ in $\Sigma^*$ with infinitesimal matrix $\bm Q^*$.  

There is a technical issue with continuous-time Markov chains, namely that jumps from a state to itself are not recognized as jumps.  In our original model with state space $\Sigma$, when a winner wins or a loser loses, there is a jump from the current state to itself in the discrete-time model and no change of state in the continuous-time model.  By augmenting the state space to $\Sigma^*$, this problem is minimized, but when a winner wins and plays again at the next turn or when a loser loses and plays again at the next turn, we have a jump from a state $(\bm x,i)$ to itself in the discrete-time model and no change of state in the continuous-time model.  These ``fictitious'' jumps are needed to define the cumulative-profit process, in terms of which the Parrondo effect is defined.  By representing the continuous-time Markov chain in terms of a Poisson process as above, these fictitious jumps become recognizable.

To be more explicit, the above representation allows us to think of the continuous-time Markov chain as having jumps from state $(\bm x,i)$ at exponential rate 1 as follows:  If $x_i=0$, the process jumps to state $(\bm x^i,j)$ with probability $N^{-1}p_{m_i(\bm x)}$ and to state $(\bm x,j)$ with probability $N^{-1}q_{m_i(\bm x)}$;  if $x_i=1$, it jumps to state $(\bm x^i,j)$ with probability $N^{-1}q_{m_i(\bm x)}$ and to state $(\bm x,j)$ with probability $N^{-1}p_{m_i(\bm x)}$.  Fictitious jumps from state $(\bm x,i)$ to itself occur with probability $N^{-1}q_{m_i(\bm x)}$ if $x_i=0$ and with probability $N^{-1}p_{m_i(\bm x)}$ if $x_i=1$.  

Let $S_n$ be the cumulative profit to the ensemble of $N$ players after $n$ steps of the discrete-time model.  Theorem~\ref{SLLN-thm} tells us that $n^{-1}S_n\to\mu^N$ a.s.  Then $S_{N(t)}$ is the cumulative profit at time $t$ to the ensemble of $N$ players in the continuous-time model, and we have
$$
{S_{N(t)}\over t}={S_{N(t)}\over N(t)}\;{N(t)\over t}\to\mu^N\;\;\text{a.s.}
$$
as $t\to\infty$, where $\mu^N$ is the mean profit per turn in the discrete-time model.  Let us summarize this result as follows.

\begin{theorem}
Assume that the conditions for ergodicity in Lemma~\ref{ergodicity-P} are satisfied.  Let $\hat S_t$ be the cumulative profit at time $t$ to the ensemble of $N$ players, assuming the continuous-time Markov chain model in $\Sigma^*$ with infinitesimal matrix $\bm Q^*:=\bm P^*-\bm I$, where $\bm P^*$ is the one-step transition matrix given by (\ref{P^*part1}) and (\ref{P^*part2}).  Then $\lim_{t\to\infty}t^{-1}\hat S_t=\mu^N$ {\rm a.s.}, where $\mu^N$ is as in Theorem~\ref{SLLN-thm}.
\end{theorem}

Implicitly, we have been discussing game $B$.  Game $A$ is the special case $p_0=p_1=p_2=p_3=1/2$.  By the discussion at the beginning of this section with $\gamma=1/2$, game $C$ is simply game $B$ with $p_m$ replaced by $(1/2+p_m)/2$ for $m=0,1,2,3$.  Thus, the Parrondo effect appears in continuous time if and only if it appears in discrete time.

\section{A spin system}\label{spin}

We want to show that our discrete-time Markov chain converges in distribution, after rescaling its time parameter, to an interacting particle system, or more specifically a \textit{spin system} on the one-dimensional integer lattice ${\bf Z}$.  Let us begin by characterizing the limiting process in terms of its generator.  Its state space is the product space
\begin{equation*}
\{0,1\}^{\bf Z}:=\{\bm x=(\ldots,x_{-2},x_{-1},x_0,x_1,x_2,\ldots): x_i\in\{0,1\}{\rm\ for\ all\ }i\in{\bf Z}\}.
\end{equation*}
We will usually refer to $x_i$ as the status (loser or winner, 0 or 1) of player $i$; occasionally, it will be convenient to refer to it as the spin at site $i$.  Let $m_i(\bm x):=2x_{i-1}+x_{i+1}$ as before but without the periodic boundary conditions.  Also, let $\bm x^i$ be the element of $\{0,1\}^{\bf Z}$ equal to $\bm x$ except at the $i$th component.  For example, $\bm x^0:=(\ldots,x_{-2},x_{-1},1-x_0,x_1,x_2,\ldots)$.

The generator depends on the four probability parameters $p_0,p_1,p_2,p_3\in[0,1]$, and it has the form
\begin{eqnarray*}
(\mathscr{L}f)(\bm x)&:=&\sum_{i\in{\bf Z}:x_i=0}p_{m_i(\bm x)}[f(\bm x^i)-f(\bm x)]+\sum_{i\in{\bf Z}:x_i=1}q_{m_i(\bm x)}[f(\bm x^i)-f(\bm x)]\\
&\;=&\sum_{i\in{\bf Z}}c_i(\bm x)[f(\bm x^i)-f(\bm x)]
\end{eqnarray*}
for functions $f$ depending on only finitely many components, where the \textit{flip rates} are given by
\begin{equation}\label{rates}
c_i(\bm x):=\begin{cases}p_{m_i(\bm x)}&\text{if $x_i=0$,}\\q_{m_i(\bm x)}&\text{if $x_i=1$,}
\end{cases}
\end{equation}
and $q_m:=1-p_m$ for $m=0,1,2,3$.  This is of the form (III.0.5) of Liggett (1985).  The sufficient condition (III.0.3) of Liggett for the characterization of the process is, in our notation,
$$
\sup_{i\in{\bf Z}}\sum_{j\in{\bf Z}}\sup_{\bm x\in\{0,1\}^{\bf Z}}|c_i(\bm x)-c_i(\bm x^j)|<\infty,
$$
which is trivially satisfied because the summands are 0 unless $|i-j|\le1$, hence the expression on the left is at most 3.  It can be shown that the functions depending on only finitely many components form a core for the generator of the Feller semigroup associated with the process.

Next we would like to justify the claim that this spin system is the limit in distribution of the $N$-player model after an appropriate time change.  First, it is convenient to relabel the $N$ players.  Instead of labeling them from 1 to $N$, we label them from $-(N-1)/2$ to $(N-1)/2$ if $N$ is odd, and from $-N/2$ to $N/2-1$ if $N$ is even.  In general, we label the players from $l_N$ to $r_N$, where
$$
l_N:=\begin{cases}-(N-1)/2&\text{if $N$ is odd,}\\
-N/2&\text{if $N$ is even,}
\end{cases}\;\;\text{and}\;\;
r_N:=\begin{cases}(N-1)/2&\text{if $N$ is odd,}\\
N/2-1&\text{if $N$ is even,}\\
\end{cases}
$$
with the understanding that players $l_N$ and $r_N$ are nearest neighbors.  The state space is
\begin{equation*}
\Sigma_N:=\{\bm x=(x_{l_N},\ldots,x_{-1},x_0,x_1,\ldots,x_{r_N}): x_i\in\{0,1\}{\rm\ for\ }i=l_N,\ldots,r_N\}.
\end{equation*}
(This is what we previously called $\Sigma$ but with the players relabeled.  To avoid confusion, we make the dependence on $N$ explicit in the notation.)
We also speed up time in the $N$-player model so that $N$ one-step transitions occur per unit of time.  The resulting discrete generator has the form
\begin{eqnarray*}
(\mathscr{L}_Nf)(\bm x)&:=&N\E[f(\bm X_N(1))-f(\bm x)\mid \bm X_N(0)=\bm x]\\
&\;=&N\bigg\{\sum_{l_N\le i\le r_N:x_i=0}N^{-1}p_{m_i(\bm x)}[f(\bm x^i)-f(\bm x)]\\
&&\qquad\qquad{}+\sum_{l_N\le i\le r_N:x_i=1}N^{-1}q_{m_i(\bm x)}[f(\bm x^i)-f(\bm x)]\bigg\},
\end{eqnarray*}
where $x_{l_N-1}:=x_{r_N}$ and $x_{r_N+1}:=x_{l_N}$.
Consequently, if we define $\eta_N:B(\{0,1\}^{\bf Z})\mapsto B(\Sigma_N)$ by 
$$
(\eta_N f)(x_{l_N},\ldots,x_{r_N}):=f(\ldots,0,0,x_{l_N},\ldots,x_{r_N},0,0,\ldots),
$$
then $(\mathscr{L}_N\eta_Nf)(\bm x)=\eta_N(\mathscr{L}f)(\bm x)$ for all $\bm x\in\Sigma_N$ and $N\ge2K+4$, where $f(\bm x)$ depends on $\bm x$ only through the $2K+1$ components $x_i$, $-K\le i\le K$.

This implies that the process $\{\bm X_N(\lfloor Nt\rfloor)\}$ converges in distribution to the spin system $\{\bm X(t)\}$ (e.g., Theorems 1.6.5 and 4.2.6 of Ethier and Kurtz 1986).  More importantly, it shows that, if the spin system has a unique stationary distribution, then the unique stationary distribution of the $N$-player Markov chain (assumed ergodic in the sense of Lemma~\ref{ergodicity-P}), converges to it in the topology of weak convergence (essentially Proposition I.2.14 of Liggett 1985).  Let us assume that the spin system has a unique stationary distribution $\pi$, and let us denote the unique stationary distribution of the $N$-player Markov chain by $\pi^N$.  (We previously denoted the latter by $\bm\pi$ but now it is necessary to make the dependence on $N$ explicit.  We do not use boldface for $\pi^N$ or $\pi$ because it is no longer useful or possible, respectively, to think of them as row vectors.)  Let us denote their $-1,1$ two-dimensional marginals by $\pi_{-1,1}^N$ and $\pi_{-1,1}$.  Then we have
\begin{eqnarray}\label{limit}
&&\pi^N_{-1,1}(0,0)p_0+\pi^N_{-1,1}(0,1)p_1+\pi^N_{-1,1}(1,0)p_2+\pi^N_{-1,1}(1,1)p_3\\
&&\qquad{}\to\pi_{-1,1}(0,0)p_0+\pi_{-1,1}(0,1)p_1+\pi_{-1,1}(1,0)p_2+\pi_{-1,1}(1,1)p_3,\nonumber
\end{eqnarray}
hence $\mu^N$, the mean parameter in Theorem~\ref{SLLN-thm}, converges as $N\to\infty$ to a limit that can be expressed in terms of the spin system.  It remains to show that this limit can be interpreted as a mean profit for the spin system; we return to this point in Section \ref{SLLN-spin}.

Under what conditions does the spin system have a unique stationary distribution (also referred to as a unique invariant probability measure)?  We will give sufficient conditions for the spin system to be \textit{ergodic}, which means not only that there is a unique stationary distribution $\pi$ but that the process at time $t$ converges in distribution to $\pi$ as $t\to\infty$, regardless of the initial distribution.

\begin{theorem}\label{ergodicity}
With $p_0,p_1,p_2,p_3\in[0,1]$, the spin system on ${\bf Z}$ with flip rates (\ref{rates}) is ergodic if at least one of the following four conditions is satisfied:

\emph{(a)} (basic estimate applies)
\begin{equation}\label{cond-a}
\max(|p_0-p_1|,|p_2-p_3|)+\max(|p_0-p_2|,|p_1-p_3|)<1;
\end{equation}

\emph{(b)} (attractiveness or repulsiveness applies)
\begin{equation}\label{cond-b}
0<\min(p_0,p_3)\le \min(p_1,p_2)\le\max(p_1,p_2)\le \max(p_0,p_3)<1;
\end{equation}

\emph{(c)} (coalescing duality applies)
\begin{equation}\label{cond-c}
\;\;\max(p_1,p_2,p_3,p_1+p_2-p_3)-p_3<p_0/2<\min(p_1,p_2,p_3,p_1+p_2-p_3);
\end{equation}

\emph{(d)} (annihilating duality applies) 
\begin{equation}\label{cond-d}
p_0,p_1,p_2,p_3\in(2\overline{p}-1,2\overline{p})\cap(0,1),\quad \overline{p}:=(p_0+p_1+p_2+p_3)/4.
\end{equation}
\end{theorem}

\begin{remark}
($i$) Absorbing states are compatible with ergodicity as long as there is only one of them.  In cases of two or more absorbing states, none of the conditions (a), (b), (c), or (d) is satisfied.  There are six such cases: If $(p_0,p_3)=(0,1)$, then $\bm0$ and $\bm1$ are absorbing; if $(p_0,p_3)=(1,0)$, then the two states with alternating $0$s and $1$s are absorbing; if $(p_0,p_1,p_2)=(1,0,0)$, then the three states of the form $(\cdots001001001\cdots)$ are absorbing; if $(p_1,p_2,p_3)=(1,1,0)$, then the three states of the form $(\cdots011011011\cdots)$ are absorbing; if $(p_0,p_1,p_2,p_3)=(1,0,0,0)$, then states with no adjacent $1$s and no more than two consecutive $0$s are absorbing; if $(p_0,p_1,p_2,p_3)=(1,1,1,0)$, then states with no adjacent $0$s and no more than two consecutive $1$s are absorbing.  (In the last two cases, the set of absorbing states is uncountable.)

($ii$) None of the four conditions on $(p_0,p_1,p_2,p_3)$ is implied by the other three.  For example, $(1/4,1/4,1/4,0)$ satisfies (a) only, $(1/4,1/4,1/4,3/4)$ satisfies (b) only, $(1/4,3/4,3/4,1)$ satisfies (c) only, and $(1/4,3/4,3/4,1/4)$ satisfies (d) only.

($iii$)  Under the assumption that $p_1=p_2$, the (three-dimensional) volumes of the regions described by (a), (b), (c), and (d) are respectively 7/12, 1/3, 7/32, and 2/3.  The volume of the union of the four regions is 3323/4032, representing about 82.4\% of the parameter space.  (See the remark following the statement of Corollary~\ref{couple-spin} for a slight improvement.)  Without the assumption that $p_1=p_2$, the (four-dimensional) volumes of the regions described by (a), (b), (c), and (d) are respectively 7/12, 1/6, 65/384, and 2/3.  The volume of the union of the four regions could not be evaluated, but can be estimated via simulation to represent about 78.9\% of the parameter space.
\end{remark}

\begin{proof}
(a) uses condition (III.0.6) of Liggett (1985).  In our notation, that condition is
$$
\sup_{i\in{\bf Z}}\sum_{j\in{\bf Z}:j\ne i}\sup_{\bm x\in\{0,1\}^{\bf Z}}|c_i(\bm x)-c_i(\bm x^j)|<\inf_{i\in{\bf Z},\bm x\in\{0,1\}^{\bf Z}}[c_i(\bm x)+c_i(\bm x^i)].
$$
The right side of this inequality is $\min_{0\le m\le3}(p_m+q_m)=1$.  So the inequality is equivalent to (a).

(b) uses a result of Gray (1982): A translation-invariant nearest-neighbor spin system on $\bf Z$ is ergodic if it is attractive or repulsive and if it has (strictly) positive flip rates.  The nearest-neighbor property requires that, for each $i\in{\bf Z}$, the flip rate at site $i$ depend only on the spins at sites $i-1$, $i$, and $i+1$.  Translation invariance (defined precisely below) strengthens this by requiring that, for each $i\in{\bf Z}$, the flip rate at site $i$ depend only on the spins at sites $i-1$, $i$, and $i+1$ \textit{in a way that does not depend on} $i$.  Attractiveness (resp., repulsiveness) is the requirement that whenever $\bm x\le\bm y$ componentwise and $i\in{\bf Z}$,
\begin{eqnarray*}
c_i(\bm x)&\le&(\text{resp.,}\ge)\; c_i(\bm y)\;\;\text{if $x_i=y_i=0$,}\\
c_i(\bm x)&\ge&(\text{resp.,}\le)\; c_i(\bm y)\;\;\text{if $x_i=y_i=1$.}
\end{eqnarray*}
In our case, attractiveness reduces to $p_0\le \min(p_1,p_2)\le\max(p_1,p_2)\le p_3$, and repulsiveness reduces to $p_3\le \min(p_1,p_2)\le\max(p_1,p_2)\le p_0$, hence (b) suffices, since the outer inequalities ensure positive flip rates.

(c) and (d) use coalescing and annihilating duality, respectively, as described in Section~III.4 of Liggett (1985).
Since our flip rates are translation invariant and nearest neighbor, there are nine parameters necessary to specify (III.4.3) and (III.4.4) of Liggett, namely (using his notation) $c(0)=c(x)$ and $p(0,A)=p(x,x+A)$ as $A$ ranges over the eight subsets of $\{-1,0,1\}$ and furthermore each such $A$ may be augmented by including $\infty$.  The basic requirement of our spin system is that $c_i(\bm x)+c_i(\bm x^i)=1$ for all $\bm x\in\{0,1\}^{\bf Z}$ and $i\in{\bf Z}$, and this implies that three of the eight probabilities are 0 (namely, the ones corresponding to $A\cap{\bf Z}=\{-1,0\}$, $A\cap{\bf Z}=\{0,1\}$, and $A\cap{\bf Z}=\{-1,0,1\}$) and $c(0)$ is determined.  This leaves five remaining parameters, which we will denote by $z_\emp$, $z_{-1}$, $z_0$, $z_1$, and $z_{-1,1}$, the interpretation being that $z_A=p(0,A)$ if $z_A\ge0$ and $z_A=-p(0,A\cup\{\infty\})$ if $z_A<0$.  They must satisfy
\begin{equation}\label{coalescing}
\begin{split}
(1-z_0)^{-1}[1-(z_\emp+z_{-1}+z_0+z_1+z_{-1,1})]&=p_0\\
(1-z_0)^{-1}[1-(z_\emp+z_{-1}+z_0)]&=p_1\\
(1-z_0)^{-1}[1-(z_\emp+z_0+z_1)]&=p_2\\
(1-z_0)^{-1}[1-(z_\emp+z_0)]&=p_3
\end{split}
\end{equation}
in the case of coalescing duality, and 
\begin{equation}\label{annihilating}
\begin{split}
2^{-1}(1-z_0)^{-1}(1+z_\emp-z_{-1}-z_0-z_1+z_{-1,1})&=p_0\\
2^{-1}(1-z_0)^{-1}(1+z_\emp-z_{-1}-z_0+z_1-z_{-1,1})&=p_1\\
2^{-1}(1-z_0)^{-1}(1+z_\emp+z_{-1}-z_0-z_1-z_{-1,1})&=p_2\\
2^{-1}(1-z_0)^{-1}(1+z_\emp+z_{-1}-z_0+z_1+z_{-1,1})&=p_3
\end{split}
\end{equation}
in the case of annihilating duality.  In either case, if there is a solution with 
\begin{equation}\label{abs-sum}
|z_\emp|+|z_{-1}|+|z_0|+|z_1|+|z_{-1,1}|<1,
\end{equation}
then Theorem~III.5.1 of Liggett (1985) ensures the ergodicity of the spin system.

The linear systems are underdetermined, and a solution to (\ref{coalescing}) is given by
\begin{eqnarray*}
z_{-1}&=&{(p_3-p_1)z_\emp\over1-p_3},\quad
z_0=1-{z_\emp\over1-p_3},\\ 
z_1&=&{(p_3-p_2)z_\emp\over1-p_3},\quad
z_{-1,1}={[p_1+p_2-(p_0+p_3)]z_\emp\over1-p_3},
\end{eqnarray*}
provided $p_3<1$.  We must have $z_\emp>0$ to ensure $|z_0|<1$.  We can also require, without loss of generality, that $z_\emp\le1-p_3$, so that $z_0\ge0$.  Then $z_{-1}$, $z_1$, and $z_{-1,1}$ can have any combination of signs, so to evaluate the sum in (\ref{abs-sum}) we must consider eight cases.
Similarly, a solution to (\ref{annihilating}) is given by
\begin{eqnarray*}
z_{-1}&=&{[p_2+p_3-(p_0+p_1)]z_\emp\over p_0+p_1+p_2+p_3-2},\quad
z_0=1-{2z_\emp\over p_0+p_1+p_2+p_3-2},\\
z_1&=&{[p_1+p_3-(p_0+p_2)]z_\emp\over p_0+p_1+p_2+p_3-2},\quad
z_{-1,1}={[p_0+p_3-(p_1+p_2)]z_\emp\over p_0+p_1+p_2+p_3-2},
\end{eqnarray*}
provided $p_0+p_1+p_2+p_3\ne2$.  We must give $z_\emp$ the same sign as $p_0+p_1+p_2+p_3-2$ to ensure $|z_0|<1$.  We can also require, without loss of generality, that $z_\emp/(p_0+p_1+p_2+p_3-2)\le1/2$, so that $z_0\ge0$.  Then $z_\emp$, $z_{-1}$, $z_1$, and $z_{-1,1}$ can have any combination of signs, so to evaluate the sum in (\ref{abs-sum}) we must consider 16 cases.

We can now state the results.  In the case of coalescing duality, let 
$$
C_{-1}:=\{p_3\ge p_1\},\quad C_1:=\{p_3\ge p_2\},\quad C_{-1,1}:=\{p_1+p_2\ge p_0+p_3\},
$$ 
and define $C_{-1}^-$, $C_1^-$, and $C_{-1,1}^-$ to be the same sets with inequalities reversed ($\ge$ becomes $\le$). 
Then a sufficient condition for ergodicity is that $(p_0,p_1,p_2,p_3)$ belong to
\begin{eqnarray}\label{coal-result}
C&:=&[C_{-1}\cap C_1\cap C_{-1,1}\cap\{p_0>0\}]\\
&&\quad{}\cup[C_{-1}\cap C_1\cap C_{-1,1}^-\cap\{p_1+p_2>p_0/2+p_3\}]\nonumber\\
&&\quad{}\cup[C_{-1}\cap C_1^-\cap C_{-1,1}\cap\{p_0/2+p_3>p_2\}]\nonumber\\
&&\quad{}\cup[C_{-1}\cap C_1^-\cap C_{-1,1}^-\cap\{p_1>p_0/2\}]\nonumber\\
&&\quad{}\cup[C_{-1}^-\cap C_1\cap C_{-1,1}\cap\{p_0/2+p_3>p_1\}]\nonumber\\
&&\quad{}\cup[C_{-1}^-\cap C_1\cap C_{-1,1}^-\cap\{p_2>p_0/2\}]\nonumber\\
&&\quad{}\cup[C_{-1}^-\cap C_1^-\cap C_{-1,1}\cap\{p_0/2+2p_3>p_1+p_2\}]\nonumber\\
&&\quad{}\cup[C_{-1}^-\cap C_1^-\cap C_{-1,1}^-\cap\{p_3>p_0/2\}].\nonumber
\end{eqnarray}
From this it is straightforward to show that $C$ is the intersection of the eight sets that appear in braces in (\ref{coal-result}).  It follows that $C$ can be expressed as in (\ref{cond-c}).

In the case of annihilating duality, let
\begin{eqnarray*}
D_\emp&:=&\{p_0+p_1+p_2+p_3>2\},\quad D_{-1}:=\{p_2+p_3\ge p_0+p_1\},\\
D_1&:=&\{p_1+p_3\ge p_0+p_2\},\qquad\, D_{-1,1}:=\{p_0+p_3\ge p_1+p_2\},
\end{eqnarray*}
and define $D_\emp^-$, $D_{-1}^-$, $D_1^-$, and $D_{-1,1}^-$ to be the same sets with inequalities reversed ($>$ and $\ge$ become $<$ and $\le$, respectively).
Then a sufficient condition for ergodicity is that $(p_0,p_1,p_2,p_3)$ belong to
\begin{eqnarray}\label{anni-result}
D&:=&[D_\emp\cap D_{-1}\cap D_1\cap D_{-1,1}\cap\{p_3<1\}]\\
&&\quad{}\cup[D_\emp\cap D_{-1}\cap D_1\cap D_{-1,1}^-\cap\{2+p_0>p_1+p_2+p_3\}]\nonumber\\
&&\quad{}\cup[D_\emp\cap D_{-1}\cap D_1^-\cap D_{-1,1}\cap\{2+p_1>p_0+p_2+p_3\}]\nonumber\\
&&\quad{}\cup[D_\emp\cap D_{-1}\cap D_1^-\cap D_{-1,1}^-\cap\{p_2<1\}]\nonumber\\
&&\quad{}\cup[D_\emp\cap D_{-1}^-\cap D_1\cap D_{-1,1}\cap\{2+p_2>p_0+p_1+p_3\}]\nonumber\\
&&\quad{}\cup[D_\emp\cap D_{-1}^-\cap D_1\cap D_{-1,1}^-\cap\{p_1<1\}]\nonumber\\
&&\quad{}\cup[D_\emp\cap D_{-1}^-\cap D_1^-\cap D_{-1,1}\cap\{p_0<1\}]\nonumber\\
&&\quad{}\cup[D_\emp\cap D_{-1}^-\cap D_1^-\cap D_{-1,1}^-\cap\{2+p_3>p_0+p_1+p_2\}]\nonumber\\
&&\quad{}\cup[D_\emp^-\cap D_{-1}\cap D_1\cap D_{-1,1}\cap\{p_0+p_1+p_2>p_3\}]\nonumber\\
&&\quad{}\cup[D_\emp^-\cap D_{-1}\cap D_1\cap D_{-1,1}^-\cap\{p_0>0\}]\nonumber\\
&&\quad{}\cup[D_\emp^-\cap D_{-1}\cap D_1^-\cap D_{-1,1}\cap\{p_1>0\}]\nonumber\\
&&\quad{}\cup[D_\emp^-\cap D_{-1}\cap D_1^-\cap D_{-1,1}^-\cap\{p_0+p_1+p_3>p_2\}]\nonumber\\
&&\quad{}\cup[D_\emp^-\cap D_{-1}^-\cap D_1\cap D_{-1,1}\cap\{p_2>0\}]\nonumber\\
&&\quad{}\cup[D_\emp^-\cap D_{-1}^-\cap D_1\cap D_{-1,1}^-\cap\{p_0+p_2+p_3>p_1\}]\nonumber\\
&&\quad{}\cup[D_\emp^-\cap D_{-1}^-\cap D_1^-\cap D_{-1,1}\cap\{p_1+p_2+p_3>p_0\}]\nonumber\\
&&\quad{}\cup[D_\emp^-\cap D_{-1}^-\cap D_1^-\cap D_{-1,1}^-\cap\{p_3>0\}].\nonumber
\end{eqnarray}
From this it is reasonably straightforward to show that $D$ is the intersection of the 16 sets that appear in braces in (\ref{anni-result}).  It follows that $D$ can be expressed as in (\ref{cond-d}).  (The use of $D_{-1}^-$ instead of $D_{-1}^c$, for example, allows using symmetry in the derivation.)

Recall that we assumed $p_3<1$ for coalescing duality and $p_0+p_1+p_2+p_3\ne2$ for annihilating duality.  Neither assumption is necessary, as we now show.

If $p_3=1$, then we can solve (\ref{coalescing}) to get
\begin{eqnarray*}
z_\emp&=&0,\quad z_{-1}=(1-p_1)(1-z_0),\\ 
z_1&=&(1-p_2)(1-z_0),\quad z_{-1,1}=-(1+p_0-p_1-p_2)(1-z_0).
\end{eqnarray*}
The inequality (\ref{abs-sum}) holds if $|1-p_1|+|1-p_2|+|1+p_0-p_1-p_2|<1$, which requires $0<p_0<2(p_1+p_2-1)$.  But this is the same requirement as (\ref{cond-c}) with $p_3=1$.

If $p_3:=2-(p_0+p_1+p_2)$, we can solve (\ref{annihilating}) to get
\begin{eqnarray*}
z_\emp&=&0,\qquad z_0=1-{z_{-1}\over1-p_0-p_1},\\
z_1&=&{(1-p_0-p_2)z_{-1}\over1-p_0-p_1},\qquad
z_{-1,1}={(1-p_1-p_2)z_{-1}\over1-p_0-p_1},
\end{eqnarray*}
provided $p_0+p_1\ne1$.  $z_{-1}$ must have the same sign as $1-p_0-p_1$ to ensure that $|z_0|<1$.  We can also require, without loss of generality, that $z_{-1}/(1-p_0-p_1)\le1$, so that $z_0\ge0$.  Then $z_{-1}$, $z_1$, and $z_{-1,1}$ can have any combination of signs, so to evaluate the sum in (\ref{abs-sum}) we must consider eight cases.  Let 
\begin{eqnarray*}
D_{-1}:=\{p_0+p_1<1\},\quad
D_1:=\{p_0+p_2\le1\},\quad\, D_{-1,1}:=\{p_1+p_2\le1\},
\end{eqnarray*}
and define $D_{-1}^-$, $D_1^-$, and $D_{-1,1}^-$ to be the same sets with inequalities reversed.
Then a sufficient condition for ergodicity is that $(p_0,p_1,p_2,p_3)$ belong to
\begin{eqnarray}\label{anni-result-special}
D&:=&[D_{-1}\cap D_1\cap D_{-1,1}\cap\{p_0+p_1+p_2>1\}]\\
&&\quad{}\cup[D_{-1}\cap D_1\cap D_{-1,1}^-\cap\{p_0>0\}]\nonumber\\
&&\quad{}\cup[D_{-1}\cap D_1^-\cap D_{-1,1}\cap\{p_1>0\}]\nonumber\\
&&\quad{}\cup[D_{-1}\cap D_1^-\cap D_{-1,1}^-\cap\{p_2<1\}]\nonumber\\
&&\quad{}\cup[D_{-1}^-\cap D_1\cap D_{-1,1}\cap\{p_2>0\}]\nonumber\\
&&\quad{}\cup[D_{-1}^-\cap D_1\cap D_{-1,1}^-\cap\{p_1<1\}]\nonumber\\
&&\quad{}\cup[D_{-1}^-\cap D_1^-\cap D_{-1,1}\cap\{p_0<1\}]\nonumber\\
&&\quad{}\cup[D_{-1}^-\cap D_1^-\cap D_{-1,1}^-\cap\{p_0+p_1+p_2<2\}].\nonumber
\end{eqnarray}
From this it is straightforward to show that $D$ is the intersection of the eight sets that appear in braces in (\ref{anni-result-special}).  The result is that ergodicity holds if $p_0+p_1+p_2+p_3=2$, $p_0+p_1\ne1$, and $p_0,p_1,p_2,p_3\in(0,1)$.  This is consistent with (\ref{cond-d}).

Finally, we must remove the restriction that $p_0+p_1\ne1$.  So we consider the case in which $p_0+p_1=1$ and $p_2+p_3=1$.  If $p_1:=1-p_0$ and $p_2:=1-p_3$, we can solve (\ref{annihilating}) to get
\begin{eqnarray*}
z_\emp=z_{-1}=0,\quad z_1=(p_3-p_0)(1-z_0),\quad z_{-1,1}=(p_0+p_3-1)(1-z_0).
\end{eqnarray*}
With $0<z_0<1$, (\ref{abs-sum}) holds because $|p_3-p_0|+|p_0+p_3-1|<1$, assuming only that $0<p_0<1$ and $0<p_3<1$.  The result is that ergodicity holds if $p_0+p_1=1$, $p_2+p_3=1$, and $p_0,p_1,p_2,p_3\in(0,1)$.  This is also consistent with (\ref{cond-d}).
\end{proof}

\begin{remark}
The strict inequalities in (\ref{coal-result}) and (\ref{anni-result}) are the requirements for (\ref{abs-sum}) to hold in the various (8 or 16) cases.  If one of these conditions fails because equality holds instead of strict inequality, then ergodicity is assured if in addition $|z_\emp|>|z_{-1,1}|$, by virtue of Corollary~III.5.8 of Liggett (1985).  For example, consider Toral's (2001) case, $(p_0,p_1,p_2,p_3)=(1,4/25,4/25,7/10)$.  This belongs to $D_\emp\cap D_{-1}^-\cap D_1^-\cap D_{-1,1}$ (line 7 of (\ref{anni-result})) but $p_0=1$.  Here $z_{-1,1}=69z_\emp$, so the additional condition is not met, and annihilating duality is inconclusive in this case.    
\end{remark}

Again we have described game $B$ above.  Game $A$ is the special case $p_0=p_1=p_2=p_3=1/2$, and game $C$, by virtue of the discussion at the beginning of Section~\ref{SLLN-cont-time}, is game $B$ with $p_m$ replaced by $(1/2+p_m)/2$ for $m=0,1,2,3$.  We note that condition (a) often fails for game $B$ but nearly always holds for game $C$ because, if $p_m$ is replaced by $(1/2+p_m)/2$ in condition (a) for $m=0,1,2,3$, the requirement becomes
\begin{equation}\label{erg-C}
\max(|p_0-p_1|,|p_2-p_3|)+\max(|p_0-p_2|,|p_1-p_3|)<2,
\end{equation}
which nearly always holds.  This result is worth emphasizing in the form of a theorem.

\begin{theorem}\label{lim mu^N}
\emph{(a)} Assume that $(p_0,p_1,p_2,p_3)$ is such that Theorem~\ref{SLLN-thm} applies to define $\mu_B^N:=\mu^N(p_0,p_1,p_2,p_3)$ for all $N\ge3$, and that the spin system on ${\bf Z}$ with flip rates (\ref{rates}) is ergodic (see Theorem~\ref{ergodicity} for sufficient conditions).  Then $\lim_{N\to\infty}\mu_B^N$ exists.

\emph{(b)} Assume only that $(p_0,p_1,p_2,p_3)$ is not equal to $(0,0,1,0)$, $(0,1,0,0)$, $(1,0,1,1)$, or $(1,1,0,1)$.   Let $\mu_C^N:=\mu^N((1/2+p_0)/2,(1/2+p_1)/2,(1/2+p_2)/2,(1/2+p_3)/2)$, where $\mu^N$ is as in Theorem~\ref{SLLN-thm}.  Then $\lim_{N\to\infty}\mu_C^N$ exists.
\end{theorem}

\begin{remark}
In the proof of part (b), we could use condition (d) of Theorem \ref{ergodicity} in place of condition (a) of that theorem.  The result is the same.
\end{remark}

\begin{proof}
The only step that remains to be checked is that, under the assumptions in part (b), ergodicity holds when the parameter vector is $((1/2+p_0)/2,(1/2+p_1)/2,(1/2+p_2)/2,(1/2+p_3)/2)$.  Condition (\ref{erg-C}) will fail if and only if $(p_0,p_1,p_2)$ equals $(0,1,1)$ or $(1,0,0)$, $(p_0,p_1,p_3)$ or $(p_0,p_2,p_3)$ equals $(0,1,0)$ or $(1,0,1)$, or $(p_1,p_2,p_3)$ equals $(0,0,1)$ or $(1,1,0)$.  In the first case, $(p_0,p_1,p_2)=(0,1,1)$, condition (b) of Theorem~\ref{ergodicity} applies if $p_3=1$, and condition (d) applies if $0\le p_3<1$.  The second case and the last two cases are treated similarly.  As for $(p_0,p_1,p_3)=(0,1,0)$, $p_2=0$ is ruled out by assumption, and condition (d) applies if $0<p_2\le 1$.  The last three cases are treated similarly.
\end{proof}

We conclude this section by showing that the invariance under rotation (and, if $p_1=p_2$, reflection)  that we found in the $N$-player model (see Lemma~\ref{invariance} and the discussion following it) carries over to the spin system, except that rotation invariance becomes translation invariance.  We could prove this as a consequence of the weak convergence of $\pi^N$ to $\pi$, but instead we give a direct proof.

Define $\tau:{\bf Z}\mapsto{\bf Z}$ by $\tau(i):=i+1$ and for $\bm x\in\{0,1\}^{\bf Z}$ define $\bm x_\tau\in\{0,1\}^{\bf Z}$ by $(\bm x_\tau)_i:=x_{\tau(i)}$.  Finally, define
$\tau:\{0,1\}^{\bf Z}\mapsto\{0,1\}^{\bf Z}$ by $\tau(\bm x):=\bm x_\tau$.  Then (cf.\ (\ref{invariance_property})) $c_i(\bm x)=c_{\tau^{-1}(i)}(\tau(\bm x))$ for all $\bm x\in\{0,1\}^{\bf Z}$ and $i\in{\bf Z}$, and therefore
\begin{equation}\label{tau-eq}
\mathscr{L}(f\circ\tau)=(\mathscr{L}f)\circ\tau
\end{equation}
for all $f$ depending on only finitely many components.  (This can be taken as the definition of \textit{translation invariance}.)  Similarly, define $\rho:{\bf Z}\mapsto{\bf Z}$ by $\rho(i):=-i$ and for $\bm x\in\{0,1\}^{\bf Z}$ define $\bm x_\rho\in\{0,1\}^{\bf Z}$ by $(\bm x_\rho)_i:=x_{\rho(i)}$.  Finally, define $\rho:\{0,1\}^{\bf Z}\mapsto\{0,1\}^{\bf Z}$ by $\rho(\bm x):=\bm x_\rho$.  If $p_1=p_2$, then $c_i(\bm x)=c_{\rho^{-1}(i)}(\rho(\bm x))$ for all $\bm x\in\{0,1\}^{\bf Z}$ and $i\in{\bf Z}$, and therefore
\begin{equation}\label{rho-eq}
\mathscr{L}(f\circ\rho)=(\mathscr{L}f)\circ\rho
\end{equation}
for all $f$ depending on only finitely many components.  (We might call this \textit{reflection invariance}.)

\begin{lemma}\label{invariance-spin}
If the spin system on ${\bf Z}$ with flip rates (\ref{rates}) is ergodic with unique stationary distribution $\pi$, then $\pi=\pi\circ\tau^{-1}$.  Under the same assumption, if also $p_1=p_2$, then $\pi=\pi\circ\rho^{-1}$.
\end{lemma}

\begin{proof}
From (\ref{tau-eq}), we have 
$$
0=\int\mathscr{L}(f\circ\tau)\,d\pi=\int(\mathscr{L}f)\circ\tau\,d\pi=
\int\mathscr{L}f\,d(\pi\circ\tau^{-1}) 
$$
for all $f$ depending on only finitely many components, so $\pi\circ\tau^{-1}$ is stationary.  The second conclusion follows from (\ref{rho-eq}) in the same way.
\end{proof}

\begin{corollary}\label{symm-spin}
If the spin system on ${\bf Z}$ with flip rates (\ref{rates}) is ergodic with unique stationary distribution $\pi$, and if also $p_1=p_2$, then the $-1,1$ two-dimensional marginal $\pi_{-1,1}$ of $\pi$ satisfies $\pi_{-1,1}(0,1)=\pi_{-1,1}(1,0)$.
\end{corollary}

\begin{proof}
This is immediate from the second conclusion of Lemma~\ref{invariance-spin}.
\end{proof}

\section{SLLN for the spin system}\label{SLLN-spin}

The approach used in Section~\ref{SLLN-cont-time} does not apply here.  One problem is that the cumulative profit to the (infinite) ensemble of players is ill-defined (being an infinite sum of $\pm1$ terms) and changes instantaneously.  So instead we consider the cumulative profit to a single player (or to a finite set of players).  This, of course, is non-Markovian.

\begin{theorem}\label{SLLN-thm-spin}
Assume that the spin system on ${\bf Z}$ with flip rates (\ref{rates}) is ergodic with unique stationary distribution $\pi$.  Let $\hat S_t$ be the cumulative profit at time $t$ to player $0$ who, along with infinitely many other players, plays according to the spin system model.  Then $\lim_{t\to\infty}t^{-1}\hat S_t=\mu$ {\rm a.s.}, where the mean parameter $\mu$ can be expressed in terms of certain two-dimensional marginals $\cdots=\pi_{-2,0}=\pi_{-1,1}=\pi_{0,2}=\cdots$ of $\pi$ as
\begin{eqnarray}\label{mu2-spin}
\mu&=&\sum_{u=0}^1\sum_{v=0}^1\pi_{-1,1}(u,v)(p_{2u+v}-q_{2u+v})\\
&=&2[\pi_{-1,1}(0,0)p_0+\pi_{-1,1}(0,1)p_1\nonumber\\
&&\qquad\quad{}+\pi_{-1,1}(1,0)p_2+\pi_{-1,1}(1,1)p_3]-1,\nonumber
\end{eqnarray}
or in terms of the one-dimensional marginals $\cdots=\pi_{-1}=\pi_0=\pi_1=\cdots$ of $\pi$ as
\begin{equation}\label{mu3-spin}
\mu=\pi_0(1)-\pi_0(0)=2\pi_0(1)-1.
\end{equation}
\end{theorem}

\begin{remark}
$(i)$ There is nothing special about player 0, so he can be replaced by any player $i\in{\bf Z}$.  In fact, for arbitrary finite $A\subset{\bf Z}$, let $\hat S_t^A$ be the cumulative profit \textit{per player} at time $t$ to the ensemble of players belonging to $A$.  Then $\lim_{t\to\infty}t^{-1}\hat S_t^A=\mu$ {\rm a.s.}, as a consequence of the theorem.

$(ii)$  Notice that $\mu$ coincides with $\lim_{N\to\infty}\mu^N$ by (\ref{limit}) and Theorem~\ref{SLLN-thm}, so we now know what the limits in Theorem~\ref{lim mu^N} are.
\end{remark}

\begin{proof}
Let $\{\bm X(t)\}$ be a stationary version of the spin system, that is, $\bm X(0)$ has distribution $\pi$.  Then, by the ergodic theorem for stationary processes,
$$
\lim_{t\to\infty}{1\over t}\int_0^t f(\bm X(s))\,ds\to\int_{\{0,1\}^{\bf Z}} f(\bm x)\,\pi(d\bm x)\;\;\text{a.s.}
$$
for all $f\in L^1(\pi)$.  In particular, with $f$ being the indicator of $\{\bm x\in\{0,1\}^{\bf Z}: x_{-1}=u,\,x_1=v\}$, we have
$$
\lim_{t\to\infty}{|\{0\le s\le t: X_{-1}(s)=u,\,X_1(s)=v\}|\over t}=\pi_{-1,1}(u,v)
$$
for all $u,v\in\{0,1\}$, where $|\cdot|$ denotes Lebesgue measure.

Now let us consider an interval $[\tau_0,\tau_1)$ such that $X_{-1}(s)=u$ and $X_1(s)=v$ for all $s\in[\tau_0,\tau_1)$.  Conditional on this, $\{X_0(s)\}$ is a continuous-time Markov chain in $\{0,1\}$ that jumps from 0 to 1 at rate $p_{2u+v}$ and from 1 to 0 at rate $q_{2u+v}$.  We can augment this process with fictitious jumps as follows:  Consider the discrete-time Markov chain $\{Y(n)\}$ in $\{0,1\}$ with one-step transition matrix
\begin{equation}\label{X_0-trans}
\setlength{\arraycolsep}{1.5mm}
\left(\begin{array}{cc}q_{2u+v}&p_{2u+v}\\q_{2u+v}&p_{2u+v}\end{array}\right),
\end{equation}
and let $\{N(t)\}$ be an independent, rate 1, Poisson process.  Then $\{Y(N(s))\}$ is the continuous-time Markov chain $\{X_0(s)\}$ with fictitious jumps whenever player 0 wins as a winner or loses as a loser.  It allows us to determine the cumulative profit of player 0 over the course of the time interval $[\tau_0,\tau_1)$.  The cumulative profit process has the following form:  Let $S_n-S_{n-1}$ be $\pm1$ depending on the $n$th jump of $\{Y(n)\}$.  (A jump to 0, from either 0 or 1, means a profit of $-1$, while a jump to 1, from either 0 or 1, means a profit of 1.)  Then $n^{-1}S_n\to p_{2u+v}-q_{2u+v}$ a.s.\ by the i.i.d.\ SLLN, so
\begin{equation}\label{slln}
{S_{N(t)}\over t}={S_{N(t)}\over N(t)}\;{N(t)\over t}\to p_{2u+v}-q_{2u+v}\;\;\text{a.s.}
\end{equation}

This assumes that the interval $[\tau_0,\tau_1)$ on which $X_{-1}(s)=u$ and $X_1(s)=v$ is the interval $[0,\infty)$.  Of course, it is not.  The set of times $s$ for which $X_{-1}(s)=u$ and $X_1(s)=v$ is a union of random intervals of the form $[\tau_0,\tau_1)$, which we may write as
$$
[\tau_0,\tau_1)\cup[\tau_2,\tau_3)\cup[\tau_4,\tau_5)\cup\cdots,
$$
where $0\le \tau_0<\tau_1<\tau_2<\cdots$.  We claim that, if these intervals are shifted to the left so as to fill up the interval $[0,\infty)$, then the profit process behaves as in (\ref{slln}).  This is due to the lack-of-memory property of the exponential distribution and to the fact that the two rows of (\ref{X_0-trans}) are equal.  

Now we consider the cumulative profit process for player 0.  Let $\hat S_t$ be cumulative profit by time $t$, and for $u,v\in\{0,1\}$ let $\hat S_t^{(u,v)}$ be cumulative profit by time $t$ \textit{accumulated at times} $s\in[0,t)$ for which $X_{-1}(s)=u$ and $X_1(s)=v$.  Then
\begin{eqnarray*}
{\hat S_t\over t}&=&\sum_{u=0}^1\sum_{v=0}^1 {\hat S_t^{(u,v)}\over t}\\
&=&\sum_{u=0}^1\sum_{v=0}^1 {|\{0\le s\le t: X_{-1}(s)=u,\,X_1(s)=v\}|\over t}\\
&&\qquad\qquad\qquad\qquad{}\cdot{\hat S_t^{(u,v)}\over|\{0\le s\le t: X_{-1}(s)=u,\,X_1(s)=v\}|}\\
&\to&\sum_{u=0}^1\sum_{v=0}^1\pi_{-1,1}(u,v)(p_{2u+v}-q_{2u+v})\\
&=&2[\pi_{-1,1}(0,0)p_0+\pi_{-1,1}(0,1)p_1+\pi_{-1,1}(1,0)p_2+\pi_{-1,1}(1,1)p_3]-1\\
&=&\mu\;\text{ a.s.},
\end{eqnarray*}
as claimed.  We can now deduce (\ref{mu3-spin}) from (\ref{mu3}).
\end{proof}

Next we generalize Corollary~\ref{couple} to our spin system.

\begin{corollary}\label{couple-spin} 
Let us refer to the spin system on ${\bf Z}$ with flip rates (\ref{rates}) as the spin system with parameter vector $(p_0,p_1,p_2,p_3)$, and let us denote $\mu$ of Theorem~\ref{SLLN-thm-spin} by $\mu(p_0,p_1,p_2,p_3)$ to emphasize its dependence on the parameter vector.

\emph{(a)} If the spin system with parameter vector $(p_0,p_1,p_2,p_3)$ is ergodic, then the spin system with parameter vector $(q_3,q_2,q_1,q_0)$ is ergodic, and 
$$
\mu(p_0,p_1,p_2,p_3)=-\mu(q_3,q_2,q_1,q_0).
$$
In particular, if $p_0+p_3=1$ and $p_1+p_2=1$, then $\mu(p_0,p_1,p_2,p_3)=0$.

\emph{(b)} If the spin system with parameter vector $(p_0,p_1,p_2,p_3)$ is ergodic, then
$$
2\min(p_0,p_1,p_2,p_3)-1\le\mu(p_0,p_1,p_2,p_3)\le2\max(p_0,p_1,p_2,p_3)-1.
$$
\end{corollary}

\begin{remark}
Three of the four regions in Theorem~\ref{ergodicity} are invariant under the transformation $\Lambda(p_0,p_1,p_2,p_3):=(q_3,q_2,q_1,q_0)$, but region $C$ corresponding to coalescing duality is not.  When $p_1=p_2$, we find that the (three-dimensional) volume of $C\cup\Lambda(C)$ is 251/720, which is nearly 60\% larger than the volume of $C$ itself.  However, the volume of the union of the five regions (the four regions of Theorem~\ref{ergodicity} and the transformed region $C$) is only 557/672, representing about 82.9\% of the parameter space.  When we do not assume that $p_1=p_2$, the (four-dimensional) volume of the union of the five regions can be estimated to represent about 79.3\% of the parameter space.
\end{remark}

\begin{proof}
Define the function $\eta:\{0,1\}^{\bf Z}\mapsto\{0,1\}^{\bf Z}$ by $\eta(\bm x):=(\ldots,1-x_{-1},1-x_0,1-x_1,\ldots)$.  Let $\{\bm X(t)\}$ be the spin system on ${\bf Z}$ with initial state $\bm x$ and flip rates $c_i^{p_0,p_1,p_2,p_3}(\bm x)$ given by (\ref{rates}), where the superscripts are merely intended to emphasize the parameter vector.  Then $\bm X'(t):=\eta(\bm X(t))$ defines a spin system $\{\bm X'(t)\}$ on ${\bf Z}$ with initial state $\bm x':=\eta(\bm x)$ and flip rates $c_i^{q_3,q_2,q_1,q_0}(\bm x')$.  To justify this assertion, it suffices to check that
$$
c_i^{p_0,p_1,p_2,p_3}(\bm x)=c_i^{q_3,q_2,q_1,q_0}(\eta(\bm x)),\qquad i\in{\bf Z},\;\bm x\in\{0,1\}^{\bf Z},
$$
hence that $\mathscr{L}^{p_0,p_1,p_2,p_3}(f\circ\eta)=(\mathscr{L}^{q_3,q_2,q_1,q_0}f)\circ\eta$.  The first conclusion of part (a) follows.  We can now take limits in Corollary~\ref{couple} to obtain the second conclusion of part (a) as well as the conclusion of part (b). 
\end{proof}

\section{The case $p_0=1$, $p_3=0$ for the spin system}\label{p0=1,p3=0-spin}

Suppose $p_0=1$ and $p_3=0$.  Let $S$ be the countable subset of $\{0,1\}^{\bf Z}$ consisting of states with alternating 0s and 1s with the single exception of a pair of adjacent 0s or a pair of adjacent 1s.  In the special case in which $p_1=p_2=1/2$, if the initial distribution is concentrated on $S$, the spin system on ${\bf Z}$ with flip rates (\ref{rates}) coincides with one studied in Example VIII.1.48 of Liggett (1985) in connection with the exclusion process.  To justify this assertion, observe that states with three or more consecutive 0s or three or more consecutive 1s are inaccessible by our spin system, and these are the only states at which the flip rates differ in the two spin systems.

If we assume that $p_1=p_2\in(1/2,1)$, then we can confirm the Parrondo effect for our spin system, at least if the initial distribution is concentrated on $S$ and if we consider the profit to a finite even number of consecutive players.

\begin{theorem}\label{p0=1,p3=0-thm-spin}
Let $p_0=1$, $p_1=p_2\in(1/2,1)$, and $p_3=0$.  Let $\mu^N_B$ (resp., $\mu^N_C$) denote the mean profit per player per unit of time to an ensemble of $N\ge1$ consecutive players always playing game $B$ (resp., $C$) in the spin system on ${\bf Z}$ with flip rates (\ref{rates}), assuming the initial distribution is concentrated on $S$.  Then $\mu^N_B=0$ for all even $N\ge2$, $\mu^N_B$ does not exist for all odd $N\ge1$, and $\mu^N_C>0$ for all $N\ge1$.  In particular, the Parrondo effect is present if and only if $N$ is even.
\end{theorem}

\begin{remark}
$(i)$ The quantities $\mu_B^N$ and $\mu_C^N$ have different meanings from those in Theorem~\ref{p0=1,p3=0-thm}.  They now represent mean profit per player per unit of time in the spin system model, whereas before they represented mean profit per turn (not per player) in the discrete-time Markov chain model.  Notice that the Parrondo effect fails for odd $N$ for a different reason here than it does in Theorem~\ref{p0=1,p3=0-thm}. 

$(ii)$ Under the assumptions of the theorem, the model has the following interpretation (stated in such a way that it can be generalized to higher dimensions):  Voters are located at the sites of ${\bf Z}$.  Each has one of two possible positions on a particular issue.  At exponential rate 1, each voter surveys his nearest neighbors and if there is a clear majority position on the issue among them, he adopts the opposite position.  If, however, there is an even split on the issue among his nearest neighbors, he tosses a biased coin to determine his position.  We might refer to the resulting spin system as the \textit{biased contrarian voter model}.
\end{remark}

We will need two lemmas in the proof.  We state them here to avoid interrupting the argument.

\begin{lemma}\label{asymptotic}
Let the random variable $T$ assume values in the set of even positive integers with $f(n):=\P(T=n)$ satisfying $f(n)\sim cn^{-\beta}$ as $n\to\infty$ through the even integers, for some constants $c>0$ and $\beta\in(1,2)$.  Let $\Delta_1,\Delta_2,\ldots$ be an independent sequence of exponential random variables with the odd-numbered terms having parameter $\lambda_1>0$ and the even-numbered terms having parameter $\lambda_2>0$, and let $\{N(t)\}$ be a rate 1 Poisson process.  Assume that $T$, $\{\Delta_n\}$, and $\{N(t)\}$ are independent, and define $X:=N(\Delta_1+\Delta_2+\cdots+\Delta_T)$.  Then $\P(X\ge k)\sim Ck^{-\beta+1}$ as $k\to\infty$ for some constant $C>0$.
\end{lemma}

\begin{proof}
Letting $\{N_j(t)\}$ be independent rate 1 Poisson processes $(j\ge1)$, we find by conditioning that
\begin{eqnarray*}
&&\!\!\!\!\!\P(X\ge k\mid T=2n)\\
&&\;{}=\P(N(\Delta_1+\Delta_2+\cdots+\Delta_{2n})\ge k)\\
&&\;{}=\P\bigg(\sum_{j=1}^{2n}[N(\Delta_1+\cdots+\Delta_j)-N(\Delta_1+\cdots+\Delta_{j-1})]\ge k)\\
&&\;{}=\P\bigg(\sum_{j=1}^n [N_{2j-1}(\Delta_{2j-1})+N_{2j}(\Delta_{2j})]\ge k\bigg),\qquad k\ge0,\;n\ge1.
\end{eqnarray*}
Now if $\Delta$ is exponential with parameter $\lambda$ and independent of the rate 1 Poisson process $\{N(t)\}$, we have
\begin{eqnarray*}
P(N(\Delta)=m)&=&\int_0^\infty {t^m\over m!}e^{-t}\lambda e^{-\lambda t}\,dt\\
&=&{\lambda\over m!}\int_0^\infty t^m e^{-(1+\lambda)t}\,dt=\bigg({\lambda\over1+\lambda}\bigg)\bigg({1\over1+\lambda}\bigg)^m,\qquad m\ge0,
\end{eqnarray*}
which is (nonnegative) geometric$(\lambda/(1+\lambda))$.  We conclude that
$$
P(X\ge k)=\sum_{n=1}^\infty f(2n)\P(X\ge k\mid T=2n)=\sum_{n=1}^\infty f(2n)\P(S_n\ge k),
$$
where $S_n:=X_1+X_2+\cdots+X_n$ and $X_1,X_2,\ldots$ are i.i.d.\ with common distribution equal to that of the sum of two independent random variables, one geometric$(\lambda_1/(1+\lambda_1))$ and the other geometric$(\lambda_2/(1+\lambda_2))$.  In particular,
$X_1$ has mean and variance
$$
\mu:={1\over\lambda_1}+{1\over\lambda_2}\quad\text{and}\quad\sigma^2:={1+\lambda_1\over\lambda_1^2}+{1+\lambda_2\over\lambda_2^2}.
$$

Let $\alpha_k:=k^\gamma$, where $\gamma\in(1/2,1)$, and observe that
\begin{eqnarray*}
&&\sum_{n\ge \lceil\mu^{-1}k+\alpha_k\rceil}f(2n)\P(S_{\lceil\mu^{-1}k+\alpha_k\rceil}\ge k)\\
&&\qquad{}\le \sum_{n\ge1}f(2n)\P(S_n\ge k)\le\P(S_{\lfloor\mu^{-1}k-\alpha_k\rfloor}\ge k)+\sum_{n\ge \lfloor\mu^{-1}k-\alpha_k\rfloor}f(2n).
\end{eqnarray*}
We can estimate these probabilities using Chebyshev's inequality:
\begin{eqnarray*}
&&\P(S_{\lceil\mu^{-1}k+\alpha_k\rceil}\ge k)\\
&&\qquad{}=1-\P(S_{\lceil\mu^{-1}k+\alpha_k\rceil}-{\lceil\mu^{-1}k+\alpha_k\rceil}\mu<k-{\lceil\mu^{-1}k+\alpha_k\rceil}\mu)\\
&&\qquad{}\ge1-\P(|S_{\lceil\mu^{-1}k+\alpha_k\rceil}-{\lceil\mu^{-1}k+\alpha_k\rceil}\mu|>|k-{\lceil\mu^{-1}k+\alpha_k\rceil}\mu|)\\
&&\qquad{}\ge1-{\sigma^2\lceil\mu^{-1}k+\alpha_k\rceil\over(k-{\lceil\mu^{-1}k+\alpha_k\rceil}\mu)^2}\\
&&\qquad{}=1-O(k^{-2\gamma+1})=1-o(1),
\end{eqnarray*}
and, if also $\gamma>\beta/2$,
\begin{eqnarray*}
&&\P(S_{\lfloor\mu^{-1}k-\alpha_k\rfloor}\ge k)\\
&&\qquad{}=\P(S_{\lfloor\mu^{-1}k-\alpha_k\rfloor}-{\lfloor\mu^{-1}k-\alpha_k\rfloor}\mu\ge k-{\lfloor\mu^{-1}k-\alpha_k\rfloor}\mu)\\
&&\qquad{}\le\P(|S_{\lfloor\mu^{-1}k-\alpha_k\rfloor}-{\lfloor\mu^{-1}k-\alpha_k\rfloor}\mu|\ge |k-{\lfloor\mu^{-1}k-\alpha_k\rfloor}\mu|)\\
&&\qquad{}\le{\sigma^2\lfloor\mu^{-1}k-\alpha_k\rfloor\over(k-{\lfloor\mu^{-1}k-\alpha_k\rfloor}\mu)^2}\\
&&\qquad{}=O(k^{-2\gamma+1})=o(k^{-\beta+1}).
\end{eqnarray*}
Finally, using
\begin{eqnarray*}
{2^{-\beta}m^{-\beta+1}\over\beta-1}&=&\int_m^\infty(2x)^{-\beta}dx\\
&\le&\sum_{n\ge m}(2n)^{-\beta}\le\int_{m-1}^\infty(2x)^{-\beta}dx={2^{-\beta}(m-1)^{-\beta+1}\over\beta-1},
\end{eqnarray*}
we find that, given $\varepsilon>0$, we have for sufficiently large $k$,
\begin{eqnarray*}
&&{(1-\varepsilon)c\,2^{-\beta}(\lceil\mu^{-1}k+\alpha_k\rceil)^{-\beta+1}\over(\beta-1)k^{-\beta+1}}(1-o(1))\\
&&\qquad{}\le {\P(X\ge k)\over k^{-\beta+1}}\le o(1)+{(1+\varepsilon)c\,2^{-\beta}(\lfloor\mu^{-1}k-\alpha_k\rfloor-1)^{-\beta+1}\over(\beta-1)k^{-\beta+1}},
\end{eqnarray*}
and our conclusion follows with $C:=c\,2^{-\beta}\mu^{\beta-1}/(\beta-1)$.
\end{proof}

\begin{lemma}\label{Berkes}
Let $X_1,X_2,\ldots$ and $X_1',X_2',\ldots$ be two sequences of random variables, all of which are nonnegative and i.i.d., with $\P(X_1\ge t)\sim Ct^{-1/2}$ as $t\to\infty$ for some $C>0$.  For each $n\ge1$, let $S_n := X_1+\cdots+X_n$ and $S_n':=X_1'+\cdots+X_n'$.  Then $0=\liminf_{n\to\infty}S_n/S_n'<\limsup_{n\to\infty}S_n/S_n'=\infty$ {\rm a.s.}
\end{lemma}

\begin{remark}
We do not know whether the assumption that $\P(X_1\ge t)\sim Ct^{-1/2}$ as $t\to\infty$ for some $C>0$ can be weakened to $E[X_1]=\infty$.
\end{remark}

\begin{proof}
First, by Kolmogorov's 0-1 law, there is a constant $A$ such that \linebreak $\limsup_{n\to\infty}S_n/S_n'=A$ a.s.  It follows that $\liminf_{n\to\infty}S_n/S_n'=1/A$ a.s., where $1/\infty:=0$.  Therefore, $A\ge1$ and it suffices to prove that $A=\infty$.

The distribution of $X_1$ is in the domain of normal attraction of the one-sided stable law of index $\alpha=1/2$ by the concluding remark of Feller (1971, p.\ 581).  This means that $S_n/n^2$ converges in distribution to the one-sided stable law of index $\alpha=1/2$.  (When the index $\alpha$ is less than 1, there is no need for centering.)  Since $X_1\ge0$, all that is needed to verify this is to show that $t^{1/2}\P(X_1\ge t)\to C$ as $t\to\infty$, which is what we assumed.  (Here $p=1$ and $q=0$ in Feller's formulation.)  The concluding remark ultimately depends on Feller's (1971, p.\ 578) Corollary~2, but notice that ``exponent $\alpha$'' in its statement should read ``exponent $-\alpha$''.  

Now the limit distribution has a positive density by Lemma~4.1 of Kanter (1975).  So we have $S_n/n^2\to_d V$ and $S_n'/n^2\to_d V'$, with $V$ and $V'$ i.i.d.  Hence $(S_n/n^2,S_n'/n^2)\to_d (V,V')$, implying $S_n/S_n'\to_d V/V'$.  But $V/V'$ also has a positive density, so given $K\ge1$, we have $\limsup_{n\to\infty}\P(S_n/S_n'>K)\ge\P(V/V'>K)>0$.  Hence $\P(S_n/S_n'>K\;{\rm i.o.})=\P(\limsup_{n\to\infty}\{S_n/S_n'>K\})\ge\limsup_{n\to\infty}\P(S_n/S_n'>K)>0$.  By Kolmogorov's 0-1 law, we have $\P(S_n/S_n'>K\;{\rm i.o.})=1$, hence $\P(\limsup_{n\to\infty}S_n/S_n'\ge K)=1$.  Since $K$ was arbitrary, $A=\infty$.
\end{proof}

\begin{proof}[Proof of Theorem~\ref{p0=1,p3=0-thm-spin}]
First, we attempt to compute $\mu_B^N$ assuming only $p_0=1$, $p_1=p_2\in(0,1)$, and $p_3=0$.

There are two absorbing states, the two states with alternating 0s and 1s.  Motivated by the proof of Theorem~\ref{p0=1,p3=0-thm}, let us assume that the initial distribution is concentrated on $S$, defined at the beginning of this section.  Let us label the states of $S$ as follows.  A state with a pair of adjacent 0s is labeled by $(0,i)$ if the leftmost 0 of the pair occurs at site $i$.  A state with a pair of adjacent 1s is labeled by $(1,i)$ if the leftmost 1 of the pair occurs at site $i$.  Then the spin system is a countable-state continuous-time Markov chain in $S$ with the following transition rates.  From $(0,i)$ the process jumps to $(1,i+1)$ at rate $p_1$ and to $(1,i-1)$ at rate $p_2$.  From $(1,i)$ it jumps to $(0,i+1)$ at rate $q_2$ and to $(0,i-1)$ at rate $q_1$.  There are two closed irreducible sets, $S^0$ consisting of states of the form $(0,\text{even})$ and $(1,\text{odd})$ and $S^1$ consisting of states of the form $(0,\text{odd})$ and $(1,\text{even})$.  The process behaves as a simple random walk with jump rates that are periodic with period 2.  From even states (i.e., those whose second component is even) in $S^0$ and odd states in $S^1$ it jumps to the right at rate $p_1$ and to the left at rate $p_2$; from odd states in $S^0$ and even states in $S^1$ it jumps to the right at rate $q_2$ and to the left at rate $q_1$.  If $p_1=p_2$, then $q_1=q_2$ and the walk is a simple symmetric random walk, but with jump rates $2p_1$ at even states in $S^0$ and odd states in $S^1$ and $2q_1$ at odd states in $S^0$ and even states in $S^1$.  (The notion of periodicity here is different from the one in the setting of discrete-time Markov chains; it merely means that the jump rates for our simple symmetric random walk are not spatially homogeneous but depend on the parity of the state.)

Let us consider what happens to players $1,2,\ldots,N$ with $N$ even.  Odd-numbered players will win every game they play and even-numbered players will lose every game they play if $(x_0,x_1,\ldots,x_N,x_{N+1})=(0,1,\ldots,0,1)$, hence if the walk starts in $S^0$ and its second component $i$ satisfies $i\ge N+1$, or if it starts in $S^1$ and $i\le-1$.  Odd-numbered players will lose every game they play and even-numbered players will win every game they play if $(x_0,x_1,\ldots,x_N,x_{N+1})=(1,0,\ldots,1,0)$, hence if the walk starts in $S^1$ and $i\ge N+1$ or if it starts in $S^0$ and $i\le-1$.  The result is that, on excursions of the random walk away from $0,1,\ldots,N$, either odd-numbered players among $1,2,\ldots,N$ will win every game they play and even-numbered players will lose every game they play, or the opposite will occur.  In any case the sequence of wins and losses to players $1,2,\ldots,N$ during one of these excursions can be modeled as a compound Poisson process.  Indeed, let $S_n:=\xi_1+\cdots+\xi_n$, where $\xi_1,\xi_2,\ldots$ are i.i.d.\ with $\P(\xi_1=1)=\P(\xi_1=-1)=1/2$, and let $N(t)$ be an independent, rate $N$, Poisson process.  Then $S_{N(t)}$ represents the cumulative profit process to the ensemble of $N$ players, since each player plays at rate 1.  It follows that $\lim_{t\to\infty}t^{-1}S_{N(t)}=0$ a.s.  To conclude that $\mu_B^N=0$, it suffices to show that the error incurred by using $t^{-1}S_{N(t)}$ instead of the actual $t^{-1}\hat S_t$ is negligible in an almost sure sense.  The error is bounded by 2$t^{-1}$ times the number of jumps by the Poisson process up to time $t$ and during the time in which the continuous-time random walk is visiting $\{0,1,\ldots,N\}$.  If we denote the set of such times by $I_t$, then our bound is
\begin{eqnarray*}
&&2t^{-1}(\text{number of jumps by $\{N(t)\}$ during $I_t$})\\
&&\qquad\qquad{}=2\,{\text{number of jumps by $\{N(t)\}$ during $I_t$}\over|I_t|}\,{|I_t|\over t}.
\end{eqnarray*}
The first fraction converges to $N$ a.s., while the second tends to 0 a.s.  The latter assertion follows from Kesten's (1965) law of the iterated logarithm for the local time of simple symmetric random walk in discrete time.  Let $L(n,i):=|\{1\le k\le n: S_k=i\}|$ and $L(n)=\sup_{i\in{\bf Z}} L(n,i)$.  Then $\limsup_{n\to\infty}L(n)/\sqrt{\,2n\log\log n}=1$ a.s.  The analogous result holds for simple symmetric random walks in continuous time by the SLLN for i.i.d.\ exponential random variables, and this suffices for our simple symmetric random walk with alternating jump rates, since the ratio $\max(p_1,q_1)/\min(p_1,q_1)$ is finite.  We conclude finally that $\mu_B^N=0$.

Next we consider the case of $N$ odd.  If also $N\ge3$, let $\hat S_t^{1,\ldots,N-1}$ and $\hat S_t^N$ denote the cumulative profit per player to players $1,\ldots,N-1$ and to player $N$, up to time $t$.  The previous paragraph tells us that $\lim_{t\to\infty}t^{-1}\hat S_t^{1,\ldots,N-1}=0$ a.s.  It is therefore enough to consider $\lim_{t\to\infty}t^{-1}\hat S_t^N$.  Equivalently, we can assume without loss of generality that $N=1$, so we focus on player 1.

Player 1 will win every game he plays if $(x_0,x_1,x_2)=(0,1,0)$, hence if the walk starts in $S^0$ and its second component $i$ satisfies $i\ge 2$, or if it starts in $S^1$ and $i\le-1$; player 1 will lose every game he plays if $(x_0,x_1,x_2)=(1,0,1)$, hence if the walk starts in $S^1$ and $i\ge 2$ or if it starts in $S^0$ and $i\le-1$.  The result is that, on excursions of the random walk away from $\{0,1\}$, either player 1 will win every game he plays or player 1 will lose every game he plays.  Define $0=\tau_0\le\tau_1<\tau_2<\tau_3<\cdots$ in terms of the discrete-time simple symmetric random walk as follows.  $\tau_1$ is the first time the walk visits $2$;  $\tau_2$ is the next time the walk visits $-2$; $\tau_3$ is the next time the walk visits $2$; $\tau_4$ is the next time the walk visits $-2$; and so on.  Then $T_k:=\tau_k-\tau_{k-1}$ are independent for $k\ge1$ and identically distributed for $k\ge2$, assuming values in the set of even positive integers; their well-known (Feller 1968, p.\ 89) distribution $f(n):=\P(T_2=n)$ satisfies $f(n)\sim cn^{-3/2}$ as $n\to\infty$ through the even integers for some $c>0$, so Lemma~\ref{asymptotic} applies.  To get player 1's profit, we must convert these discrete-time excursions to continuous time and then run a Poisson process with rate 1 over those intervals, alternating between always winning and always losing.  With $\hat S_t$ denoting the cumulative profit to player 1 up to time $t$, we claim that $-1=\liminf_{t\to\infty}t^{-1}\hat S_t<\limsup_{t\to\infty}t^{-1}\hat S_t=1$ a.s., so that, with probability 1, the limit does not exist and hence $\mu_B^N$ does not exist.  To prove this we need only consider a subsequence of times that achieves these limits.  Using the result of Kesten described above, we can assume without loss of generality that on $[\tau_1,\tau_2)$, $[\tau_3,\tau_4)$, and so on, player 1 always wins, while on $[\tau_2,\tau_3)$, $[\tau_4,\tau_5)$, and so on, player 1 always loses.  (This assumes the walk starts in $S^0$; if it starts in $S^1$, the opposite occurs.)  We also assume, again without loss of generality, that $\tau_1=\tau_0=0$, that is, we ignore the time before the initial excursion.  Let $X_1,X_2,\ldots$ be the profits of player 1 accumulated during the winning excursions, and let $X_1',X_2',\ldots$ be the losses of player 1 accumulated during the losing excursions.  The two sequences of random variables are i.i.d.\ and distributed as $X$ in Lemma~\ref{asymptotic} with $\beta=3/2$, $\lambda_1=2p_1$, and $\lambda_2=2q_1$.  Denote partial sums by $S_n:=X_1+\cdots+X_n$ and $S_n':=X_1'+\cdots+X_n'$ for each $n\ge1$.  Then, along a particular subsequence of times, $t^{-1}\hat S_t$ is asymptotic to
$$
{X_1-X_1'+X_2-X_2'+\cdots+X_n-X_n'\over X_1+X_1'+X_2+X_2'+\cdots+X_n+X_n'}={S_n-S_n'\over S_n+S_n'}={S_n/S_n'-1\over S_n/S_n'+1}.
$$
By Lemma~\ref{Berkes}, this has limit inferior equal to $-1$ a.s. and limit superior equal to 1 a.s., as required.

This proves the assertions about $\mu_B^N$.  It remains to show that $\mu_C^N>0$ when $p_1=p_2\in(1/2,1)$.

Now $\mu_C^N=\mu(3/4,(1/2+p_1)/2,(1/2+p_1)/2,1/4)$, which by Corollary~\ref{couple-spin} is 0 at $p_1=1/2$.  If we could show that this function is increasing in $p_1$, the proof would be complete.  However, this monotonicity property appears difficult to prove.  As in the proof of Theorem~\ref{p0=1,p3=0-thm}, we take an alternative approach.

Let $\pi_{-1,1}$ be the $-1,1$ two-dimensional marginal of the stationary distribution $\pi$ when the probability parameters are $3/4$, $(1/2+p_1)/2$, $(1/2+p_1)/2$, and $1/4$.  We apply Theorem~\ref{SLLN-thm-spin} twice.  By (\ref{mu3-spin}) and Corollary~\ref{symm-spin},
\begin{eqnarray}\label{z-x}
\mu_C^N&=&\pi_0(1)-\pi_0(0)=\pi_{-1}(1)-\pi_{-1}(0)\\
&=&\pi_{-1,1}(1,0)+\pi_{-1,1}(1,1)-[\pi_{-1,1}(0,0)+\pi_{-1,1}(0,1)]\nonumber\\
&=&\pi_{-1,1}(1,1)-\pi_{-1,1}(0,0).\nonumber
\end{eqnarray}
By (\ref{mu2-spin}), Corollary~\ref{symm-spin}, and (\ref{z-x}),
\begin{eqnarray*}
\mu_C^N&=&\pi_{-1,1}(0,0)(1/2)+\pi_{-1,1}(0,1)(2p_1-1)+\pi_{-1,1}(1,1)(-1/2)\\
&=&(2p_1-1)\pi_{-1,1}(0,1)-(1/2)\mu_C^N.
\end{eqnarray*}
Therefore, $\mu_C^N=(2/3)(2p_1-1)\pi_{-1,1}(0,1)$.  

We claim that $\pi_{-1,1}(0,1)>0$.  Suppose not.  Then, by Corollary~\ref{symm-spin}, $\pi_{-1,1}(0,1)=\pi_{-1,1}(1,0)=0$, so by Lemma~\ref{invariance-spin}, $\pi$ is concentrated on the four states that have the same spin at every odd site and the same spin at every even site.  Let us denote those four states by $\bm 0$ (all 0s), $\bm 1$ (all 1s), $\bm1_{\text{even}}$ (1s at even sites only), and $\bm1_{\text{odd}}$ (1s at odd sites only).  Then there exist nonnegative $a_0,a_1,a_2,a_3$ with $a_0+a_1+a_2+a_3=1$ such that $\pi=a_0\delta_{\bm 0}+a_1\delta_{\bm 1}+a_2\delta_{\bm1_{\text{even}}}+a_3\delta_{\bm1_{\text{odd}}}$.  But this leads to a contradiction.  We apply $\int\mathscr{L}f\,d\pi=0$ with $f(\bm x):=x_0$, $f(\bm x):=x_{-1}x_1$, and $f(\bm x):=x_0x_2$, to get four linear equations in $a_0$, $a_1$, $a_2$, and $a_3$, namely
$3a_0-3a_1-a_2+a_3=0$, 
$-6a_1-2a_3=0$, 
$-6a_1-2a_2=0$, and
$a_0+a_1+a_2+a_3=1$, 
for which there is no solution with $a_0,a_1,a_2,a_3\ge0$.  Since $p_1>1/2$, we conclude that $\mu_C^N>0$.
\end{proof}

\section{Open problems}\label{open}

Suppose, for example, that $(p_0,p_1,p_2,p_3)=(1/10,3/5,3/5,3/4)$.  Then the spin system on ${\bf Z}$ with flip rates (\ref{rates}) is attractive, so by Theorem~\ref{ergodicity}, it is ergodic.  (In fact, conditions (c) and (d) of Theorem~\ref{ergodicity} also apply.)  Hence by Theorem~\ref{lim mu^N} we know that $\lim_{N\to\infty}\mu_B^N$ and $\lim_{N\to\infty}\mu_C^N$ exist.  We have computed these means exactly for $3\le N\le 19$ (Ethier and Lee 2012) and have found that the former has stabilized to 4 significant digits by $N=19$, while the latter has stabilized to 11 significant digits by $N=19$.  In fact, in both cases, the 17 computed numbers are monotonically increasing (why?).  Also, $\mu_B^N<-1/500$ for $3\le N\le19$ and $\mu_C^N>1/100$ for $4\le N\le19$.  Can we say anything about the limits, based on these computations?  Can we conclude that the Parrondo effect is present for all $N\ge4$?  Or that it is present in the spin system?

The monotonicity problem mentioned in the proof of Theorem~\ref{p0=1,p3=0-thm} is interesting.  More generally, assuming that the Markov chain is ergodic, let $\mu^N(p_0,p_1,p_2,p_3)$ be the mean parameter of Theorem~\ref{SLLN-thm}.  Is this function monotonically increasing in each variable?  A partial result in this direction is that it is nondecreasing in each variable if we restrict to parameter vectors satisfying the attractiveness condition, $p_0\le \min(p_1,p_2)\le \max(p_1,p_2)\le p_3$.  (This follows from Theorem~III.1.5 of Liggett 1985.)  Of course, the same question can be asked about the limiting mean function.

Can ergodicity of the spin system be established more generally?  For example, is it sufficient that $p_0,p_1,p_2,p_3\in(0,1)$?  It seems likely that Theorem~\ref{ergodicity} can be strengthened, perhaps via another form of duality.

We have considered only the one-dimensional integer lattice.  Mihailovi\'c and Rajkovi\'c (2006) studied the two-dimensional $MN$-player model with periodic boundary conditions, assuming that the bias of the coin tossed depends only on the number (not the set) of winners among the four nearest neighbors.  Using computer simulation they found evidence of the Parrondo effect for parameters close to those of the voter model.  What can be proved in dimensions greater than one?

\section*{Acknowledgments}

The authors thank Istv\'an Berkes for Lemma~\ref{Berkes}.

The research for this paper was carried out during S.N.E.'s visit to the Department of Statistics at Yeungnam University in 2011--12.  He is grateful to his hosts for their hospitality.

\end{document}